\documentclass[12pt]{amsart}

\usepackage{amssymb}
\usepackage[shortlabels]{enumitem}
\usepackage{hyperref}
\usepackage{thmtools}
\usepackage{xcolor}

\usepackage{interval}
\intervalconfig{soft  open  fences}

\makeatletter
\@namedef{subjclassname@2020}{%
  \textup{2020} Mathematics Subject Classification}
\makeatother

\declaretheorem{theorem}
\declaretheorem[sibling=theorem]{proposition}
\declaretheorem[sibling=theorem]{lemma}

\declaretheorem[sibling=theorem]{claim}

\declaretheorem[sibling=theorem]{corollary}
\declaretheorem[style=definition,sibling=theorem]{definition}

\declaretheorem[style=remark,sibling=theorem]{remark}

\DeclareMathOperator{\cf}{cf}
\DeclareMathOperator{\cof}{cof} 
\DeclareMathOperator{\ot}{ot}

\DeclareMathOperator{\dom}{dom}

\DeclareMathOperator{\range}{range}
\DeclareMathOperator{\Lim}{Lim}

\renewcommand{\th}{^\text{th}}

\newcommand{\seq}[2]{\langle #1 : #2 \rangle}
\newcommand{\pair}[2]{\langle #1 , #2 \rangle}

\newcommand{\Col}{\textup{Col}}

\newcommand{\ON}{\textup{ON}}

\newcommand{\CH}{\textup{\textsf{CH}}}
\newcommand{\ZFC}{\textup{\textsf{ZFC}}}

\newcommand{\then}{\implies}

\newcommand{\rest}{\upharpoonright}

\newcommand{\intv}{\interval}
\newcommand{\nrest}{\!\rest\!}

\renewcommand{\P}{\mathbb P}

\newcommand{\T}{\mathbb T}
\renewcommand{\S}{\mathbb S}

\newcommand{\ed}{\textup{\textsf{ed}}}
\newcommand{\ep}{\textup{\textsf{ep}}}
\newcommand{\fd}{\textup{\textsf{fd}}}
\newcommand{\fp}{\textup{\textsf{fp}}}

\DeclareMathOperator{\thr}{thr}
\DeclareMathOperator{\pos}{pos}

\DeclareMathOperator{\limp}{lim^+}
\newcommand{\bit}[4]{(\pair{#1}{#2},\intv{#3}{#4})}

\DeclareMathOperator{\out}{out}
\DeclareMathOperator{\Coh}{Coh}

\begin{document}

\title{Unthreadability with small conditions}
\author[M. Levine]{Maxwell Levine}
\address{Albert-Ludwigs-Universit\"at Freiburg,
Mathematisches Institut, Abteilung f\"ur math. Logik, Ernst--Zermelo--Stra\ss e~1, 
79104 Freiburg im Breisgau, Germany}
\email{maxwell.levine@mathematik.uni-freiburg.de}

\begin{abstract} We introduce a forcing that adds a $\square(\aleph_2,\aleph_0)$-sequence with countable conditions under $\CH$. Assuming the consistency of a weakly compact cardinal, we can find a forcing extension by our new poset in which both $\square(\aleph_2,<\!\aleph_0)$ and $\square_{\aleph_1,\aleph_0}$ fail in the forcing extension.\end{abstract}

\keywords{Forcing, large cardinals}

\maketitle

\section{Introduction}

The tension between the stationary reflection principles of large cardinals and the fine-structural combinatorics that hold in canonical inner models is a prominent theme in set theory. In particular, variations of Jensen's principle $\square_\kappa$, which holds for all cardinals $\kappa$ in G{\"o}del's Constructible Universe $L$, have been studied widely. The construction of new models realizing some compatibility between large cardinal properties and square properties can help us develop a more vivid picture of the individual cardinals and their relationships to one another.

The principle $\square_\kappa$ and its relatives are known as \emph{square principles}. Each such principle asserts that there exists a coherent sequence $\mathcal{C} = \seq{\mathcal{C}_\alpha}{\alpha < \lambda}$ where $\mathcal{C}_\alpha$ is a set of closed unbounded sets in $\alpha$, and that $\mathcal{C}$ does not have a \emph{thread}, meaning that there is no club $D \subseteq \lambda$ such that $D \cap \alpha \in \mathcal{C}_\alpha$ for all limit points $\alpha \in D$. Square principles of the form $\square_{\mu,\kappa}$ assert the existence of such a $\mathcal{C}$ of length $\mu^+$ (notice the discrepancy in notation) in which the members of the $\mathcal{C}_\alpha$'s have order-type bounded by $\mu$, which serves as the reason why there cannot be a thread. On the other hand, square principles of the form $\square(\mu^+,\kappa)$ assert the existence of such a $\mathcal{C}$ of length $\mu^+$ where the non-existence of the thread is simply declared outright. The consequences of this distinction are significant for this paper.

Although a varieties of square principles have been studied thoroughly, somewhat less is known about the varieties of forcing extensions in which certain square principles hold. This paper will follow a strain of research having to do with adding square sequences of length $\mu^+$ with forcings consisting of conditions of cardinality $<\mu$. The poset introduced here adds a $\square(\aleph_2,\aleph_0)$-sequence with countable conditions. This is different from the conventional way of forcing $\square(\mu^+,\kappa)$-sequences, which uses conditions of cardinality $\mu$. This approach was introduced by Jensen and studied thoroughly by Cummings, Foreman, and Magidor. We will refer to this as \emph{Jensen's method} for forcing square sequences, since we will have reason to compare and contrast it with the approach used here, especially regarding the behavior of the complementary threading forcing.


The idea of the forcing described in this paper is roughly based on Baumgartner's forcing to add a club to $\aleph_1$ with finite sequences, except that we use its generalization to $\aleph_2$, and we crucially use a presentation that is due to Abraham \cite{Abraham-Shelah1983}.  More precisely, one of the main hurdles for adding a new object of size $\aleph_2$ with countable conditions is finding a way to preserve $\aleph_2$ knowing that that new $\aleph_1$-sequences will be added. In our case we use elementary submodels to guarantee preservation of $\aleph_2$. This use of elementary submodels is alluded to by Mitchell when he discusses the options for adding a new closed unbounded set to $\aleph_2$ in terms of possibilities for the sizes of conditions \cite{Mitchell2005}. It also distinguishes the forcing presented in this paper from the Shelah-Stanley poset for adding $\square_{\aleph_1}$ with countable conditions, which uses a chain condition to preserve the cardinal that is turned in to $\aleph_2$ \cite{Stanley-Shelah1982}. (The Shelah-Stanley poset adds a strictly stronger square sequence as well.) The technical difficulty comes mostly in showing that the new forcing is countably closed, or to be more precise, that it has a dense countably closed subset. This property also contrasts the new forcing with the forcings for adding $\square_{\aleph_1}$ using finite conditions originally used by Dolinar and D{\v z}amonja \cite{Dolinar-Dzamonja2013} and later streamlined by Neeman \cite{Neeman2017} since those cannot be countably closed. Above all, the motivation for the work here is to provide a new and distinct model in which a much-studied combinatorial principle holds.

\begin{theorem}\label{maintheorem} If $V \models \CH$ then there is a forcing $\S$ consisting of countable conditions such that $V[\S] \models \square(\aleph_2,\aleph_0)$. Moreover, if $W \models ``\kappa$ is weakly compact'', then $W[\Col(\aleph_1,<\kappa)][\S] \models \neg \square(\aleph_2,<\aleph_0) \wedge \neg \square_{\aleph_1,\aleph_0}$.\end{theorem}

We expect that the construction presented in this paper generalizes to higher cardinals to get a $\square(\kappa^{++},\kappa)$-sequence under the assumption that $2^\kappa = \kappa^+$. The assumption of a weakly compact cardinal is optimal because work of Jensen and Todor{\v c}evi{\' c} shows that this is required to make $\square(\kappa,1)$ fail for regular $\kappa$ \cite{Todorcevic1987}.

This paper is organized as follows: The remainder of the introduction will cover background. We will introduce $\S$ in the second section and show that it adds a $\square(\aleph_2,\aleph_0)$-sequence. The third section will introduce the new threading forcing $\T$ and show how it can be used with large cardinals to make stronger square principles fail.

We assume familiarity with the technique of forcing \cite{Jech2003}. As for our conventions: Given a set of ordinals $X$, $\Lim(X)$ is the set of limit ordinals in $X$, $\lim(X) = \{\alpha \in X:X \cap \alpha \text{ is unbounded}\}$, $\limp(X)=\{\alpha \le \sup(X):X \cap \alpha \text{ is unbounded}\}$, and $j[X] = \{j(\alpha):\alpha \in X\}$ if $j$ is a function. We also let $\ot(X)$ denote the order-type of $X$, and $X \cap \cof(\tau) = \{\alpha \in X: \cf(\alpha) = \tau\}$. If $\P$ is a forcing poset and $p,q \in \P$, then $p \le q$ means that $p$ has more information than $q$. If we say that $\P$ is $\kappa$-distributive for a regular $\kappa$, we mean that $\P$ does not add new functions $f:\lambda \to \ON$ for $\lambda < \kappa$, and when we say that $\P$ is $\kappa$-closed, we mean that it is closed under $\le_\P$-decreasing sequences of length $\lambda$ for all $\lambda<\kappa$. The notation $V[\P]$ will refer to an extension by an unspecified generic. When depicting an iteration $\P_0 \ast \dot{\P}_1$, we will often drop the dot and write $\P_0 \ast \P_1$.


\subsection{Definitions of Squares}\label{subsec-defs}

First we define square principles of the form $\square(\lambda,\kappa)$. These are examples of incompactness principles, because they imply the non-existence of an object that appears to be approximated, namely a thread.\footnote{See a survey of Cummings for background on square principles of the form $\square_{\mu,\kappa}$ and stationary reflection \cite{Cummings2005}.}


\begin{definition} Let $\kappa$ and $\lambda$ be regular cardinals such that $\kappa < \lambda$. Then $\seq{\mathcal{C}_\alpha}{\alpha \in \Lim(\lambda)}$ is a \emph{$\square(\lambda,\kappa)$-sequence} if the following hold for all $\alpha \in \Lim(\lambda)$:

\begin{enumerate}
\item $\mathcal{C}_\alpha$ consists of clubs in $\alpha$.
\item $1 \le |\mathcal{C}_\alpha| \le \kappa$.
\item For all $C \in \mathcal{C}_\alpha$ and all $\beta \in \lim(C)$, $C \cap \beta \in \mathcal{C}_\beta$.
\item There is no club $D \subset \lambda$ such that for all $\alpha \in \lim(D)$, $D \cap \alpha \in \mathcal{C}_\alpha$.
\end{enumerate}

If there is a $\square(\lambda,\kappa)$-sequence, then we say that \emph{$\square(\lambda,\kappa)$ holds}. If $V \models ``\vec{\mathcal{C}}$ is a $\square(\lambda,\kappa)$ sequence'' and $W \supset V$ is a model containing some $D$ such that $D \cap \alpha \in \mathcal{C}_\alpha$ for all $\alpha \in \lim(D)$, then we say that $D$ \emph{threads} $\vec{\mathcal{C}}$.
\end{definition}

\begin{definition} Let $\kappa$ and $\mu$ be cardinals. Then $\seq{\mathcal{C}_\alpha}{\alpha \in \Lim(\mu^+)}$ is a \emph{$\square_{\mu,\kappa}$-sequence} if is a $\square(\mu^+,\kappa)$-sequence, and moreover for all $\alpha \in \Lim(\mu^+)$ and $C \in \mathcal{C}_\alpha$, $\ot C \le \mu$.\end{definition}

\begin{remark}\label{whynothread} If $\mathcal{C}$ is a $\square_{\mu,\kappa}$-sequence, then there is no $D \subset \mu^+$ such that for all $\alpha \in \lim(D)$, $D \cap \alpha \in \mathcal{C}_\alpha$. (Hence point (4) from the definition of $\square(\lambda,\kappa)$-sequences is implied.) This is because if $\gamma$ where the $(\kappa+\omega)\th$ point of $D$, then $D \cap \gamma \in \mathcal{C}_\gamma$ has an order-type larger than $\kappa$.\end{remark}

\begin{definition} If $\mathcal{C}$ is a $\square(\lambda,\kappa)$-sequence, we say that a poset $\P$ \emph{threads} $\mathcal{C}$ if $\P$ forces that there is a thread of $\mathcal{C}$.\end{definition}

\subsection{Elementary Submodels and Generic Conditions}\label{subsec-submodels}

In this subsection we will outline the basic arguments we use to preserve $\aleph_2$. In this regard, there are a couple of important conventions we will use:

\begin{itemize}

\item In the context of discussing a poset $\P$, we will refer to a regular cardinal as ``sufficiently large'' when $H(\Theta)$ contains enough information about $\P$ for the argument at hand. In this paper we mean that $H(\Theta)$ contains $\P$ and all of its antichains. More generally, we will refer to a structure $K \models \ZFC-\textup{\textsf{Powerset}}$ as ``sufficiently rich'' when it contains $\P$ and its antichains.

\item Given a poset $\P$ and a sufficiently rich structure $K$, we will refer to an elementary submodel $M \prec K$ as \emph{basic} if $\P \in M$, $|M|=\aleph_1$, $M \cap \aleph_2 \in \aleph_2$, and $M^\omega \subseteq M$, i.e$.$ $M$ is closed under countable sequences. This usage is specific to this paper because we will use it frequently.

\end{itemize}



\begin{proposition}\label{basicprop} Suppose $\CH$ holds, $|X| = \aleph_1$, and $K$ is a sufficiently rich structure with $X,\aleph_2 \in K$. Then there is a basic model $M \prec K$ such that $X \subset M$.\end{proposition}

\begin{proof} Let $M_0 \prec K$ be an elementary submodel of cardinality $\aleph_1$ such that $X \subset M_0$. Then take a chain $\seq{M_i}{i<\aleph_1}$ of $\aleph_1$-sized elementary submodels of $K$ such that for all $i<\aleph_1$, $\sup(M_i \cap \aleph_2) \subseteq M_{i+1}$, $M_i^{\omega} \subseteq M_{i+1}$ (this is where we use $\CH$), and such that there is continuity in the sense that $M_i = \bigcup_{j<i}M_j$ for all limits $i<\aleph_1$. Then let $M = \bigcup_{i<\aleph_1}M_i$.\end{proof}

\begin{definition} Let $\P$ be a poset, let $K$ be a sufficiently rich structure, and let $M \prec K$ such that $\P \in M$.

\begin{itemize}

\item A condition $q$ is \emph{$(M,\P)$-generic} if for every maximal antichain $A \in M$ (such that $A \subseteq \P$), for all $q' \le q$, there is some $q'' \le q'$ such that $q'' \le p$ for some $p \in A \cap M$. Equivalently, $q$ is $(M,\P)$-generic if $q \Vdash ``\dot{G} \cap M$ is $\P \cap M$-generic over $M$'' \cite{Jech2003}.

\item A condition $q$ is \emph{strongly $(M,\P)$-generic} if for all $q' \le q$, there is some $p \in \P \cap M$ such that for every $p' \le p$ with $p' \in \P \cap M$, there is some $q'' \le q'$ such that $q'' \le p'$. Equivalently, $q$ is strongly $(M,\P)$-generic if $q \Vdash ``\dot{G} \cap M$ is $\P \cap M$-generic over the ground model $V$'' \cite{Mitchell2006}.

\end{itemize}
\end{definition}

It is clear that all strongly $(M,\P)$-generic conditions are $(M,\P)$-generic for sufficiently rich structures. We distinguish between these two types of generic conditions for the sake of some remarks that we will make below.

\begin{proposition}\label{master-cond-prop} Suppose $\P$ is a poset, $K$ is a sufficiently rich structure, and $M \prec K$ with $\P \in M$. Suppose $q$ is $(M,\P)$-generic and $G$ is $\P$-generic over $V$ with $q \in G$. Then the following hold:

\begin{enumerate}

\item $M[G] \prec K[G]$.

\item $V \cap M[G] = M$.

\item If $\dot f \in M$ is such that $\Vdash_\P ``\dot f: \check{\kappa}_1 \to \check{\kappa}_2$'' where $\kappa_1 \subset M$ and $\kappa_2 \in M$, then $\range(\dot{f}[G]) \subset M$.

\end{enumerate}
\end{proposition}

\begin{proof} These are standard arguments,\footnote{See the chapter on proper forcing in Jech \cite{Jech2003}.} so we will prove only \emph{(3)} since \emph{(1)} and \emph{(2)} use the same trick. Work in $V$ and let $\alpha \in \kappa_1$. Let $A \in M$ be a maximal antichain consisting of conditions deciding $\dot{f}(\check \alpha)$. By elementarity there is such an $A \in M$. Given some $q' \le q$, let $q'' \le q'$ be below some $p \in A \cap M$ using the fact that $q$ is an $(M,\P)$-generic condition. By elementarity, $p \Vdash ``\dot{f}(\check \alpha) = \check \beta$ for some $\beta \in M$, so the same information is forced by $q''$.\end{proof}

\begin{definition} Suppose $\P \in M \prec K$ for $K$ sufficiently rich. We say that $M$ \emph{has generic conditions for $\P$ (resp$.$ strongly generic conditions)} if for all $p \in \P \cap M$, there is some $q \le p$ such that $q$ is $(M,\P)$-generic (resp$.$ strongly $(M,\P)$-generic).\end{definition}

\begin{proposition} Let $M \prec K$ be a basic elementary submodel of a sufficiently rich structure such that $M$ has generic conditions for $\P$. Then $\P$ preserves $\aleph_2$.\end{proposition}

\begin{proof} Working in $V$, we have $\aleph_2 \in M$ by elementarity and $\aleph_1 \subset M$ by $M^\omega \subset M$. $M$ would contain a $\P$-name $\dot f$ for a supposed collapsing function $f:\aleph_1^V \to \aleph_2^V$ by elementarity. Apply \autoref{master-cond-prop}.\end{proof}

\section{Defining the Main Poset}\label{sec2}


The goal of this section is to introduce our poset for adding a $\square(\aleph_2,\aleph_0)$-sequence with countable conditions.



\begin{definition}\label{newforcing} Let $\mathfrak{S}$ be the set of $(\pair{\alpha}{x}, \intv{\beta}{\gamma})$ such that $\alpha < \aleph_2$ is a limit ordinal, $x \in \omega$, and $\intv{\beta}{\gamma}$ is a closed interval of ordinals such that $-1 \le \beta \le \gamma < \alpha$.\footnote{We are abusing the term ``ordinal'' in the case that $\beta = -1$.} Define $\S$ to be the set of $s \subset \mathfrak{S}$ such that the following hold:

\begin{enumerate}

\item\label{cond-size} The set $s$ is countable.


\item\label{cond-disjointness} If $(\pair{\alpha}{x},\intv{\beta}{\gamma})$, $(\pair{\alpha}{x},\intv{\beta'}{\gamma'}) \in s$, then either $\beta = \beta'$ or $\intv{\beta}{\gamma} \cap \intv{\beta'}{\gamma'} = \emptyset$.

\item\label{cond-oneomega1club} If $\cf(\alpha) = \omega_1$, then for all $x,y \in \omega$ and $\beta,\gamma<\alpha$, $(\pair{\alpha}{x},\intv{\beta}{\gamma}) \in s$ if and only if $(\pair{\alpha}{y},\intv{\beta}{\gamma}) \in s$.

\item\label{cond-pointclosure} If $\{ (\pair{\alpha}{x} , \intv{\beta_i}{\gamma_i}) : i < \omega \} \subseteq s$ and $\beta^*=\sup_{i<\omega}\beta_i < \alpha$, then
it follows that $(\pair{\alpha}{x} , \intv{\beta^*}{\gamma^*}) \in s $ for some $\gamma^*$.

\item\label{cond-domainclosure} If $\{(\pair{\alpha_i}{x_i},\intv{\beta_i}{\gamma_i}): i<\omega \} \subset s$ and $\alpha^* = \sup_{i<\omega}\alpha_i$, then it follows that $(\pair{\alpha^*}{y},\intv{\beta}{\gamma}) \in s$ for some $y,\beta,\gamma$.

\item\label{cond-unboundedness} If $(\pair{\alpha}{x},\intv{\bar \beta}{\bar \gamma}) \in s$ and $\cf(\alpha) = \omega$, then $\{\beta<\alpha: \exists \gamma, (\pair{\alpha}{x},\intv{\beta}{\gamma}) \in s \}$ is unbounded in $\alpha$.




\item\label{cond-decisiveness} If $(\pair{\alpha}{x},\intv{\beta}{\gamma}) \in s$ and $\cf(\beta)=\omega$, then one of the following holds:

\begin{enumerate}
\item there exists a sequence of ordinals $\seq{\beta_i}{i<\omega}$ converging to $\beta$ such that $\forall i < \omega, \exists \gamma_i,(\pair{\alpha}{x},\intv{\beta_i}{\gamma_i}) \in s$.
\item there is some $\bar \beta < \beta$ and some sequence of ordinals $\seq{\gamma_i}{i<\omega}$ converging to $\beta$ such that $\forall i < \omega$, $(\pair{\alpha}{x},\intv{\bar \beta}{\gamma_i}) \in s$.
\end{enumerate}

%

\item\label{cond-coherence} Suppose that $(\pair{\alpha}{x},\intv{\beta}{\gamma}) \in s$ and either:
\begin{itemize}
\item $\cf(\beta) = \omega_1$ or else 
\item $\cf(\beta)=\omega$ and $\{\beta'<\beta: \exists \gamma',(\pair{\alpha}{x},\intv{\beta'}{\gamma'}) \in s\}$ is unbounded in $\beta$.
\end{itemize}
Then there is some $y \in \omega$ such that for all $\beta',\gamma' < \beta$, $(\pair{\alpha}{x},\intv{\beta'}{\gamma'}) \in s$ if and only if $(\pair{\beta}{y},\intv{\beta'}{\gamma'}) \in s$.

\end{enumerate}


Suppose $s \in \S$, $\beta,\alpha < \aleph_2$, and $x,y \in \omega$. We write that $\Coh(s,\alpha,x,\beta,y)$ holds if both of the following hold:

\begin{enumerate}[(i)]
\item $(\pair{\alpha}{x},\intv{\beta}{\gamma}) \in s$ and either
\begin{itemize}
\item $\cf(\beta) = \omega_1$ or else
\item $\cf(\beta)=\omega$ and the set $\{\beta'<\beta: \exists \gamma',(\pair{\alpha}{x},\intv{\beta'}{\gamma'}) \in s\}$ is unbounded in $\beta$;
\end{itemize}
\item for all $\beta' \le \gamma'< \beta$, $(\pair{\alpha}{x},\intv{\beta'}{\gamma'}) \in s$ if and only if $(\pair{\beta}{y},\intv{\beta'}{\gamma'}) \in s$.
\end{enumerate}


The ordering is defined so that $s' \le_\S s$ holds if and only if:

\begin{enumerate}
\item $s' \supseteq s$.
\item If $\Coh(s,\alpha,x,\beta,y)$ holds then $\Coh(s',\alpha,x,\beta,y)$ holds.


\end{enumerate}
\end{definition}

We will also use the following conventions:

\begin{itemize}

\item An element of $\mathfrak S$ is called a \emph{bit}.

\item If $s \in \S$, the set $\{\alpha<\aleph_2: \exists x,\beta,\gamma, (\pair{\alpha}{x},\intv{\beta}{\gamma}) \in s\}$ is called the \emph{domain} of $s$ and is abbreviated $\dom(s)$.

\item If $s \in \S$, then $\dom(s)$ has a maximal element $\gamma$, which will be denoted $\max(s)$.

\item For $\pair{\alpha}{x} \in \Lim(\aleph_2) \times \omega$, we let
\[
\dot C_\alpha^x := \{(\check \beta,s):\beta \ge 0,\exists \gamma, \bit{\alpha}{x}{\beta}{\gamma} \in s \in \S \}.
\]

\item  We sometimes use the terminology of coherence before we have verified that some $s \subset \mathfrak{S}$ is in fact a condition in $\S$. The second point of the definition of $\le_\S$ states that $s' \le_\S s$ means that if $s$ forces that $y$ witnesses coherence of $\dot{C}^x_\alpha$ at $\beta$, then so does $s'$.


\item If $\delta \in (\aleph_2 \cap \cof(\omega_1)) \cup \{\aleph_2\}$ and $s \in \S$, we write $s \nrest \delta := \{(\pair{\alpha}{x},\intv{\beta}{\gamma}) \in s: \alpha < \delta \}$, noting that $s \nrest \delta \in \S$.


\item For $s \in \S$, $\out s = \{\beta \ge 0: \exists \alpha, x, \gamma, \bit{\alpha}{x}{\beta}{\gamma} \in s, s \not\Vdash ``\ot(\dot{C}_\alpha^x) = \omega"\}$.


\end{itemize}

We can describe an intuition for \autoref{newforcing} here. If $(\pair{\alpha}{x},\intv{\beta}{\gamma}) \in s \in \S$ and $\beta \ge 0$, then this indicates that $s$ forces $\beta$ to be a point in $\dot{C}_\alpha^x$ and also forces $``\dot{C}_\alpha^x \cap \interval[open left]{\beta}{\gamma} = \emptyset"$. Note that $x$ is just a placeholder label here. \autoref{cond-size} and \autoref{cond-disjointness} should be thought of in association with Abraham's presentation of Baumgartner's forcing for adding a club in $\aleph_1$ with finite conditions \cite{Abraham-Shelah1983}, but generalized to $\aleph_2$. \autoref{cond-oneomega1club} asserts that if $\cf(\alpha) = \omega_1$, then the set of $\dot{C}_\alpha^x$ for $x \in \omega$ is in fact a singleton. \autoref{cond-pointclosure} ensures that the clubs in the $\square(\aleph_2,\aleph_0)$-sequence are closed. \autoref{cond-domainclosure} ensures that we will be able to extend any condition. \autoref{cond-unboundedness} ensures that the clubs in ordinals $\alpha$ of countable cofinality are unbounded. \autoref{cond-decisiveness} in part ensures that the question of whether $\beta$ is a limit point of its club is determined by the condition, so that in closure arguments we only need to deal with freshly added limits of $\beta$'s. Finally, \autoref{cond-coherence} ensures that the generic object added by $\S$ is a coherent sequence. The second point of the definition of $\le_\S$ will be used in the proof of \autoref{closure} below.


The next steps are to establish the facts about $\S$ that do not require countable closure.

\begin{proposition}\label{basiclimit} Given $s \in \S$, $\alpha \in \dom s$, and $x \in \omega$, the following are equivalent:

\begin{enumerate}

\item $s \Vdash ``\beta \in \lim \dot{C}_\alpha^x"$.

\item Either:

\begin{enumerate}[(a)]

\item $\cf(\beta) = \omega_1$ and $\bit{\alpha}{x}{\beta}{\gamma} \in s$ for some $\gamma$, or else

\item there are sequences $\seq{\beta_i}{i<\omega}$ and $\seq{\gamma_i}{i<\omega}$ converging to $\beta$ such that $\bit{\alpha}{x}{\beta_i}{\gamma_i} \in s$ for all $i<\omega$.

\end{enumerate}
\end{enumerate}
\end{proposition}

\begin{proof} \emph{(2)$\then$(1)}: The implication is clear if (b) holds, so assume that (a) holds.

We consider the case $\cf(\alpha) = \omega$. Suppose $s' \le s$ (we are now suppressing the notation $\le_\S$) and $\bar \beta < \beta$, where we want to show that there is some $s'' \le s'$ such that $s'' \Vdash ``\interval[open right]{\bar \beta}{\beta} \cap \dot{C}_\alpha^x \ne \emptyset"$.

Let
\[
\beta^* = \sup\{\beta'<\beta:\exists \gamma', \bit{\alpha}{x}{\beta'}{\gamma'} \in s'\},
\]
and
\[
\gamma^* = \sup\{\gamma'<\beta:\exists \beta', \bit{\alpha}{x}{\beta'}{\gamma'} \in s'\}.
\]
so $\beta^* \le \gamma^* < \beta$ by the fact that $s'$ is countable and $\cf(\beta) = \omega_1$. Moreover, $\beta^* \in \dom (s')$ by \autoref{cond-pointclosure} and \autoref{cond-coherence} together.

If $\bar \beta \le \beta^*$ then we are done. Suppose otherwise. Since our goal is to obtain $s''$ such that $s'' \Vdash ``\interval[open right]{\bar \beta}{\beta} \cap \dot{C}_\alpha^x \ne \emptyset"$, we can assume a larger value of $\bar{\beta}$ without loss of generality. Therefore, we assume that $\bar{\beta} \ge \gamma^*$. Then we let
\begin{align*}
s'' = & s' \cup  \{\bit{\alpha}{x}{\bar{\beta}+1}{\bar{\beta}+1} \} \\
& \{\bit{\tilde{\alpha}}{y}{\bar{\beta}+1}{\bar{\beta}+1}): \Coh(s',\tilde{\alpha},y,\alpha,x) \textup{ holds}\}.
\end{align*}



Let us argue that $s''$ is a condition by going through the clauses.

\autoref{cond-size}: This is immediate since we are only adding countably many bits.

 \autoref{cond-disjointness}: We do not need to consider bits of the form $(\pair{\alpha'}{y},\intv{\beta'}{\gamma'})$ if $\alpha' \ne \alpha$. We also do not need to consider pairs of bits that are taken from $s'$. The clause clearly holds if both bits under consideration are the newly-added ones. The remaining cases are where one bit takes the form $(\pair{\alpha}{x},\intv{\beta'}{\gamma'}) \in s'$ and the other bit is one of the new ones, so the clause holds because we know that $\beta^* \le \gamma^* \le \bar{\beta}$.

\autoref{cond-oneomega1club}: This holds vacuously since $\cf(\alpha)=\omega$.

\autoref{cond-pointclosure}: Observe that if $\{\bit{\alpha'}{x'}{\beta_i}{\gamma_i}:i<\omega\} \subseteq s''$, then we can verify the clause by considering each $\pair{\alpha'}{x'}$ as a sub-case and noting that we have added finitely many bits for each such sub-case.


\autoref{cond-domainclosure}: This holds for the same reason as \autoref{cond-pointclosure}.

\autoref{cond-unboundedness}: This holds because $s'$ is a condition and
\[
\{\pair{\alpha}{x}:\exists \beta,\gamma,(\pair{\alpha}{x},\intv{\beta}{\gamma}) \in s'\} = \{\pair{\alpha}{x}:\exists \beta,\gamma,(\pair{\alpha}{x},\intv{\beta}{\gamma}) \in s''\}.
\]

\autoref{cond-decisiveness}: The notable case is that in which we consider a bit of the form $(\pair{\alpha'}{x'},\intv{\beta'}{\gamma'}) \in s''$ and $\beta' = \bar{\beta}+1$, in which case the clause holds vacuously. The other cases follow from the fact that $s'$ is a condition.

\autoref{cond-coherence}: This clause holds because we dealt with all $\alpha',x'$ for which $\Col(s',\tilde{\alpha},y,\alpha,x)$ holds, and we did not add new instances in which the hypothesis of \autoref{cond-coherence} holds.


Given that $s''$ is a condition, it is clear that $s'' \Vdash ``\interval[open right]{\bar \beta}{\beta} \cap \dot{C}_\alpha^x \ne \emptyset"$.

We also need to have $s'' \le_\S s'$. We have $s'' \supseteq s'$, so we then need to consider the coherence clause for $\le_\S$. This holds for reasons similar to the fact that \autoref{cond-coherence} holds in that no bits have been added where coherence needs to be verified.

The only substantive difference in the case where $\cf(\alpha)=\omega_1$ is for Clause 3, which holds because it is explicitly taken care of for all $y \in \omega$. Using the notation from the other case, we use
\begin{align*}
s'' = & s' \cup \{\bit{\alpha}{x}{\bar{\beta}+1}{\bar{\beta}+1} :x \in \omega\} \\
& \{\bit{\tilde{\alpha}}{y}{\bar{\beta}+1}{\bar{\beta}+1}): x \in \omega, \Coh(s',\tilde{\alpha},y,\alpha,x) \textup{ holds}\}.
\end{align*}

The argument that $s'' \le_\S s'$ is also analogous in this case.

\emph{(1)$\then$(2)}: Assume that $\beta$ is a limit ordinal and that \emph{(2)} does not hold. First suppose that $\cf(\alpha) = \omega$. Let
\[
\beta^* = \sup\{\beta':\beta'<\beta:\exists \gamma', \bit{\alpha}{x}{\beta'}{\gamma'} \in s'\}
\]
if this set is not empty, otherwise $\beta^* = 0$. Then the assumption that \emph{(2)} does not hold implies that $\beta^*<\beta$, either by \autoref{cond-pointclosure} or by $\cf(\beta)=\omega_1$.

If there is no $\gamma$ such that $\bit{\alpha}{x}{\beta}{\gamma} \in s$, then we let 
\[
s' = s \cup \{ \bit{\alpha}{x}{\beta^*}{\beta} \} \cup \{\bit{\tilde{\alpha}}{y}{\beta^*}{\beta}:\Coh(s',\tilde{\alpha},y,\alpha,x) \textup{ holds}\}.
\]

The argument that $s'$ is a condition and that $s' \le_\S s$ is similar to the backwards direction of the proof of this proposition, but it is strictly easier since we are only adding one bit. Then we see that $s' \Vdash `` \dot{C}^x_\alpha \cap \interval[open left]{\beta^*}{\beta} = \emptyset$'' and hence $s' \Vdash ``\beta \notin \lim \dot{C}^x_\alpha$''.


If there is some $\gamma$ such that $\bit{\alpha}{x}{\beta}{\gamma} \in s$, then $\cf(\beta)=\omega$ (since we are assuming \emph{(2)} does not hold). Then case (b) of \autoref{cond-decisiveness} holds, so there is some $\bar{\beta}<\beta$ and a sequence of $\gamma_i$'s such that $(\pair{\alpha}{x},\intv{\bar{\beta}}{\gamma_i}) \in s$ for all $i<\omega$. Hence $s \Vdash ``\dot{C}^x_\alpha \cap (\bar{\beta},\beta) = \emptyset$'' and therefore $s \Vdash ``\beta \notin \lim \dot{C}^x_\alpha$''.

An analogous argument applies for the case that $\cf(\alpha)=\omega_1$, i.e$.$ in the sense that we need only alter the argument to ensure that \autoref{cond-oneomega1club} holds.\end{proof}




\begin{proposition}\label{objects-are-clubs} The name $\dot{C}_\alpha^x$ is forced to be a closed unbounded set in $\alpha$ for all limit ordinals $\alpha<\aleph_2$ and all $x<\omega$.\end{proposition}

\begin{proof} First we show that the $\dot{C}_\alpha^x$'s are forced to be closed. (We will prove nonemptiness when we prove unboundedness.) If there are $\beta'$ and $\gamma'$ with $(\pair{\alpha}{x},\intv{\beta'}{\gamma'}) \in \bar{s}$, then $\bar{s}$ forces $\dot{C}^x_\alpha$ to be closed under countable sequences by \autoref{cond-pointclosure}. Closure under uncountable sequences follows from \autoref{basiclimit} since if $s \in \S$ forces $\beta \notin \lim \dot{C}_\alpha^x$ for some $\beta$ with $\cf(\beta)=\omega_1$, then $s$ forces $\beta \notin \dot{C}_\alpha^x$.

Most of the proof then consists of showing that the $\dot{C}_\alpha^x$'s are forced to be unbounded and in particular nonempty.

Consider the case of $\alpha$ with $\cf(\alpha)= \omega$.
 Fix some $\bar{s} \in \S$. If there are $\beta$ and $\gamma$ with $(\pair{\alpha}{x},\intv{\beta}{\gamma}) \in \bar{s}$, then $\bar{s}$ forces $\dot{C}^x_\alpha$ to be unbounded by \autoref{cond-unboundedness}. 

Now suppose there are no $\beta$ and $\gamma$ with $(\pair{\alpha}{x},\intv{\beta}{\gamma}) \in \bar{s}$. Take a sequence $\seq{\delta_i}{i<\omega}$ converging to $\alpha$ where $\delta_0 = 0$ and $\delta_i$ is a successor ordinal for $i>0$. Then let $s:= \bar{s} \cup \{(\pair{\alpha}{x},\intv{\delta_i}{\delta_{i+1}-1}): x \in X, i<\omega\}$. As long as we show that $s$ is a condition such that $s \le_\S \bar{s}$, it is clear that $s$ forces $\dot{C}_\alpha^x$ to be unbounded and (vacuously) closed in $\alpha$.

We argue that $s$ is a condition, going through the less trivial clauses:

\autoref{cond-disjointness}: The intervals $\intv{\delta_i}{\delta_{i+1}-1}$ are pairwise disjoint.

\autoref{cond-pointclosure}: This closure holds vacuously when the new bits are considered.

\autoref{cond-unboundedness}: The $\delta_i$'s are unbounded in $\alpha$.

\autoref{cond-coherence}: The only notable case is where $\bar{s} \Vdash ``\alpha \in \lim \dot{C}_{\alpha'}^{x'}$'' for some $\pair{\alpha'}{x'}$, but in this case we would already have $\Coh(\bar{s},\alpha',x',{\color{blue}\alpha},y)$ for some $y<\omega$.


The fact that $s' \le_\S s$ is quite immediate given the argument for \autoref{cond-coherence}.

Now we consider the case of $\alpha$ with $\cf(\alpha) = \omega_1$.

For the case where $\cf(\alpha)=\omega_1$. Fix $x<\omega$ and $\bar{\beta}<\alpha$ and some $\bar{s} \in \S$ such that for some $\beta',\gamma'$, $(\pair{\alpha}{x},\intv{\beta'}{\gamma'}) \in \bar{s}$.

Let
\[
\beta^* = \sup\{\beta'<\alpha:\exists \gamma', \bit{\alpha}{x}{\beta'}{\gamma'} \in \bar{s}\},
\]
so $\beta^* < \alpha$ by the fact that $s'$ is countable and $\cf(\beta) = \omega_1$. If $\beta^* \ge \bar{\beta}$ then we are done. (We can also allow the set to be empty and for $\beta^*$ to be $0$.) Otherwise, we let
\[
\gamma^* = \sup\{\gamma'<\alpha:\exists \beta', \bit{\alpha}{x}{\beta'}{\gamma'} \in \bar{s}\},
\]
and we assume without loss of generality that $\gamma^* \le \bar{\beta}$.

Now let
\begin{align*}
s = & \bar{s} \cup \{\bit{\alpha}{y}{\bar{\beta}+1}{\bar{\beta}+1}:y<\omega \} \cup \\ & \{\bit{\tilde{\alpha}}{\tilde{x}}{\bar{\beta}+1}{\bar{\beta}+1}:\Coh(\bar{s},\tilde{\alpha},\tilde{x},\alpha,y) \textup { holds}\}.
\end{align*}

The proof that $s$ is a condition such that $s \le_\S \bar{s}$ is very similar to the first part of the proof of \autoref{basiclimit}. From there it is immediate that $s \Vdash ``\dot{C}_\alpha^x \cap (\bar{\beta},\alpha) \ne \emptyset$''.\end{proof}







Now we will work towards proving that a dense subset of $\S$ is countably closed. First we define the subset:


\begin{definition} We say that $s \in \S$ is \emph{complete} if the following holds: If $(\pair{\alpha}{x},\intv{\beta}{\gamma}) \in s$ and there are $\alpha'<\alpha$ and some $x' \in \omega$ such that $(\pair{\alpha'}{x'},\intv{\beta'}{\gamma'}) \in s$, then:

\begin{itemize}
\item there are $\beta'' \le \alpha'$ and $\gamma'' \ge \alpha'$ such that $(\pair{\alpha}{x},\intv{\beta''}{\gamma''}) \in s$,
\item there are $\beta''' \le \beta'$ and $\gamma''' \ge \beta'$ such that $(\pair{\alpha}{x},\intv{\beta'''}{\gamma'''}) \in s$.
\end{itemize}
In other words $s$ decides $\alpha' \in \dot{C}_\alpha^x$ and $\beta' \in \dot{C}_\alpha^x$ either positively or negatively.
\end{definition}

\begin{proposition} The set of complete conditions is dense in $\S$. More precisely, if $\bar{s} \in \S$, then there is a complete $s \le_\S \bar{s}$ such that $\dom{s} = \dom{\bar{s}}$.\end{proposition}

\begin{proof} Fix $\bar s \in \S$. For all $\alpha \in \dom \bar{s}$ and $x<\omega$, define
\[
B = \{ \beta : \exists \alpha,x, \gamma, (\pair{\alpha}{x},\intv{\beta}{\gamma}) \in \bar{s}\}\textup{ and }A = \dom \bar{s}.
\]
and
\[
B_{\pair{\alpha}{x}} = \{ \beta : \exists \gamma, (\pair{\alpha}{x},\intv{\beta}{\gamma}) \in \bar{s}\}.
\]
If $\gamma \in A \cup B$ and $\gamma<\alpha$, let $\beta^\gamma_{\pair{\alpha}{x}} = \sup(B_{\pair{\alpha}{x}} \cap \gamma)$.  (In particular, there is some $\gamma'$ such that $\bit{\alpha}{x}{\beta^\gamma_{\pair{\alpha}{x}}}{\gamma'} \in \bar{s}$ and there are no $\tilde{\beta},\tilde{\gamma}$'s with $\beta^\gamma_{\pair{\alpha}{x}} < \tilde{\beta} \le \gamma$ and $\bit{\alpha}{x}{\tilde{\beta}}{\tilde{\gamma}} \in \bar{s}$. It is also possible that $\beta^\gamma_{\pair{\alpha}{x}}=\gamma$.)

Let
\begin{align*}
s = \bar{s} \cup \{\bit{\alpha}{x}{\beta^\gamma_{\pair{\alpha}{x}}}{\gamma}: & \exists \beta',\gamma',(\pair{\alpha}{x},\intv{\beta'}{\gamma'}) \in \bar{s},\\ & \gamma \in A \cup B, \gamma<\alpha,\alpha \in \dom(\bar{s})\}.
\end{align*}

We can see that $s$ is complete by construction. To the extent that we are adding new information, we are only forcing ordinals not to be in the $\dot{C}^x_\alpha$'s. Now we verify that $s$ is a condition.

\autoref{cond-size}: Immediate.

\autoref{cond-disjointness}: Fix $\alpha$ and $x$ for consideration. Suppose $(\pair{\alpha}{x},\intv{\beta}{\gamma}) \in s$ and $(\pair{\alpha}{x},\intv{\beta'}{\gamma'}) \in s$. If they are both in $\bar{s}$ then we are done. If they are both in $s \setminus \bar{s}$ with $\beta \le \beta'$, then by definition there are no points in $B_{\pair{\alpha}{x}}$ between $\beta$ and $\gamma$, so either $\beta = \beta'$ or $\gamma < \beta'$. If without loss of generality the first of these bits is in $\bar{s}$ and the other is in $s \setminus \bar{s}$ with $\beta \le \beta'$, then we know that there are no elements of $B_{\pair{\alpha}{x}}$ in the interval $\interval[open left]{\beta}{\gamma}$, so either $\beta = \beta'$ or $\gamma < \beta'$. If we assume that $\beta' \le \beta$ the argument is similar.

\autoref{cond-oneomega1club}: This property is inherited by $s$ from $\bar{s}$ because the same changes are made to every $\dot{C}_\alpha^x$ for each $x<\omega$.

\autoref{cond-pointclosure}: We have for all $\alpha$ and $x$ that
\[
\{\beta:\exists \gamma,(\pair{\alpha}{x},\intv{\beta}{\gamma}) \in s \} = \{\beta:\exists \gamma,(\pair{\alpha}{x},\intv{\beta}{\gamma}) \in \bar{s}\}
\]
and therefore the fact that \autoref{cond-pointclosure} holds for $s$ follows from the fact that it holds for $\bar{s}$.

\autoref{cond-domainclosure}: This holds because $\dom \bar{s} = \dom s$ and moreover the $x$'s for which $\dot{C}_\alpha^x$ is defined are the same.

\autoref{cond-unboundedness}: Holds for the same reason as \autoref{cond-domainclosure}.

\autoref{cond-decisiveness}: Holds for the same reason as \autoref{cond-pointclosure}.

\autoref{cond-coherence}: Fix $\pair{\alpha}{x}$. Suppose first that $\cf(\beta)=\omega$ and that the set $\{\beta':\exists \gamma',(\pair{\alpha}{x},\intv{\beta'}{\gamma'}) \in s\}$ is unbounded in $\beta$. Then it follows that $\{\beta':\exists \gamma',(\pair{\alpha}{x},\intv{\beta'}{\gamma'}) \in \bar{s}\}$ is unbounded in $\beta$ and therefore $\Coh(\bar{s},\alpha,x,\beta,y)$ holds for some $y<\omega$. But if we write
\[
B' = \{ \beta'<\beta :\exists  \gamma', (\pair{\beta}{y},\intv{\beta'}{\gamma'}) \in \bar{s}\}
\]
then $B \cap \beta = B'$. Moreover, $B_{\pair{\alpha}{x}} \cap \beta = B_{\pair{\beta}{y}}$. Therefore, given how the bits in $\bar{s}$ are defined, we have $\Coh(s,\alpha,x,\beta,y)$. The case where $\cf(\beta)=\omega_1$ is analogous.

Finally, see that $s \le_\S \bar{s}$, which follows because clearly $s \supseteq \bar{s}$ and by an argument analogous to the argument for \autoref{cond-coherence}.\end{proof}

The proof that the set of complete conditions in $\S$ is countably closed is similar to the analogous proof for Jensen's method of forcing square sequences in the sense that $\omega$-sequences can vacuously top off lower bounds of countable sequences. However, it is different in the sense that new limit points of the $\dot{C}_\alpha^x$'s are added. Because of this complication, we introduce some additional terminology.

\begin{definition}\label{kindsofpoints} Let $\vec s = \seq{s_i}{i<\omega}$ be an $\le_\S$-decreasing sequence of conditions.

\begin{itemize} 


\item Let $\ed(\vec s) = \bigcup_{i<\omega}\dom(s_i)$ be the \emph{existing domain}.


\item Let $\fd(\vec s) = \limp(\ed(\vec s)) \setminus \ed(\vec s)$ be the \emph{fresh domain}.




\item Let the \emph{existing points} $\ep(\vec s)$ consist of $\beta \in \aleph_2$ such that for some $\alpha \in \ed$, some $x \in \omega$, there is some $i<\omega$ such that $s_i \Vdash \beta \in \lim C^x_\alpha$.



\item Suppose $\beta \in \lim(\aleph_2)$ and $ \alpha \in \ed(\vec s)$ as witnessed by $(\pair{\alpha}{x},\intv{\beta'}{\gamma'})$. Suppose that:
\begin{itemize}
\item $\beta<\alpha$ and there is \emph{no} $i<\omega$ such that $s_i \Vdash \beta \in \lim C^x_\alpha$,
\item there are sequences $\seq{\beta_n}{n<\omega}$ and $\seq{\gamma_n}{n<\omega}$ with supremum $\beta$ such that $\forall n < \omega, (\pair{\alpha}{x},\intv{\beta_n}{\gamma_n}) \in \bigcup_{i<\omega}s_i$.
\end{itemize}
Then we let $\beta \in \fp(\vec s)$, the set of \emph{fresh points}, and we say that \emph{$\pair{\alpha}{x}$ witnesses $\beta \in \fp(\vec s)$}.

\end{itemize}


If $\bar s$ is a lower bound of $\vec s$, then we say that $\bar s$ is \emph{parsimonious} if:

\begin{itemize}

\item $\dom(\bar s) = \ed(\vec s) \cup \fd(\vec s) \cup \fp(\vec s)$;
\item for all $\alpha \in \Lim(\aleph_2)$ and $x \in \omega$, $\bit{\alpha}{x}{\beta}{\gamma} \in \bar s$ implies either $\beta \in \ep(\vec s) \cup \fp(\vec s)$ or $\bar s \Vdash ``\ot(\dot{C}_\alpha^x) = \omega"$. 
\end{itemize}

We drop the notation for $\vec s$ when the context is clear.\end{definition}


\begin{lemma}\label{closure} The set of complete conditions in $\S$ is countably closed. Moreover, every $\le_\S$-decreasing sequence $\seq{s_i}{i<\omega}$ of complete conditions has a parsimonious lower bound.\end{lemma}

Before proving the lemma, we establish the basic relationships between the sets of points described above.

\begin{proposition}\label{noninterference} Let $\vec s = \seq{s_i}{i<\omega}$ be an $\le_\S$-decreasing sequence of complete conditions in $\S$.


\begin{enumerate}


\item $\fd(\vec s),\fp(\vec s) \subset \aleph_2 \cap \cof(\omega)$.

\item $\fp(\vec s) \cap \ep(\vec s) = \emptyset$, and if $\beta \in \fp(\vec s)$, then there is no $i<\omega$ such that $s_i \Vdash \beta \notin \lim C_\alpha^x$.

\item $\fp(\vec s) \cap \ed(\vec s) = \emptyset$.

\end{enumerate}

\end{proposition}

\begin{proof} \emph{(1)} is immediate.

To prove \emph{(2)}, first note that the first equality follows by definition, since the fresh points are defined to exclude the existing points. Now suppose that $\beta \in \fp(\vec s)$ is witnessed by $\pair{\alpha}{x}$ and $\seq{\beta_n}{n<\omega}$ and $\seq{\gamma_n}{n<\omega}$. Suppose for contradiction that there is some $i<\omega$ such that $s_i \Vdash \beta \notin \lim C_\alpha^x$. Then there are $\gamma_1 < \beta \le \gamma_2$ such that $(\pair{\alpha}{x},\intv{\gamma_1}{\gamma_2}) \in \dom s_i$. Suppose that  $n$ is large enough that $\gamma_1 <\beta_n$. Then this is a contradiction.

The argument that $\fp(\vec s) \cap \ed(\vec s) = \emptyset$ is analogous: Again, suppose that $\beta \in \fp(\vec s)$ is witnessed by $\pair{\alpha}{x}$ and that $\beta \in \ed(\vec s)$. But by completeness of $s_i$ and the fact that $\beta<\alpha$, we have some $\beta'$ and $\gamma'$ such that $(\pair{\alpha}{x},\intv{\beta'}{\gamma'}) \in s_i$ and $\beta' \le \beta \le \gamma'$. But then this contradicts $\beta \in \fp(\vec s)$.\end{proof}


\begin{proof}[Proof of \autoref{closure}] Let $\seq{s_i}{i<\omega}$ be an $\le_\S$-decreasing sequence of complete conditions in $\S$ and fix $\ed$, $\fd$, $\ep$, and $\fp$ as in \autoref{kindsofpoints}. Let $S_0 := \bigcup_{i<\omega}s_i$. For each $\beta \in \fp$, let $X_\beta$ be the set of pairs $\pair{\alpha}{x}$ witnessing that $\beta \in \fp$. We let
\[
S_1 := \{(\pair{\alpha}{x},\intv{\beta}{\beta}): \alpha \in \ed, \beta \in \fp, \pair{\alpha}{x} \in X_\beta\}.
\]
For each $\beta \in \fp$, let $f_\beta:\omega \to X_\beta$ be a surjection. Let
\begin{align*}
S_2 := \{(\pair{\beta}{x},\intv{\beta'}{\gamma'}): & \beta \in \fp, f_\beta(x)=\pair{\alpha}{y} \in X_\beta,\\ & (\pair{\alpha}{y},\intv{\beta'}{\gamma'}) \in S_0, \gamma' < \beta , x \in \omega\}.
\end{align*}
Also let
\[
S_3 := \{(\pair{\beta}{x},\intv{\beta'}{\beta'}): \beta'<\beta, \fp \supseteq \{\beta,\beta' \}, f_\beta(x)=\pair{\alpha}{y} \in X_\beta \cap X_{\beta'} \}.\footnote{This means that if $\pair{\alpha}{x}$ witnesses that both $\beta \in \fp$ and $\beta' \in \fp$ where $\beta'<\beta$, we make sure to put the point $\beta'$ in $\dot{C}_{\beta}^x$ where we will show that $\Coh(s,\alpha,y,\beta,x)$ holds.}
\]
For each $\alpha \in \fd \setminus \fp$, choose a sequence $\seq{\delta^\alpha_n}{n<\omega}$ of ordinals converging to $\alpha$ where $\delta^\alpha_0 = 0$ and $\delta^\alpha_n$ is a successor ordinal for all $n>0$. Let
\[
S_4 := \{(\pair{\alpha}{x},\intv{\delta_n^\alpha}{\delta_{n+1}^\alpha-1}): \alpha \in \fd \setminus \fp, x \in \omega, n < \omega \}.
\]


The lower bound we seek is $s := S_0 \cup S_1 \cup S_2 \cup S_3 \cup S_4$. It is clear that $s$ is parsimonious as long as it is a lower bound. It remains to argue that $s$ is a condition and that $s \le_\S s_i$ for all $i<\omega$.


Assuming that $s$ is in fact a condition, it is relatively straightforward to argue that it is a lower bound: Since we have $s \supseteq s_i$ for all $i<\omega$, the first point in the definition of $\le_\S$ holds. For the second point, suppose that $\Coh(s_i,\alpha,x,\beta,y)$ holds. Then $\alpha,\beta \in \ed$. Since \autoref{cond-coherence} holds for $s_i$ we do not need to worry about bits from $S_0$, for which \autoref{cond-coherence} is witnessed by other bits from $S_0$, and so we only need to consider bits from $S_1$ since this is the only one of the $S_k$'s with bits of the form $(\pair{\alpha^*}{x^*},\intv{\beta^*}{\gamma^*})$ with $\alpha^* \in \ed$. Suppose that $(\pair{\alpha}{x},\intv{\beta'}{\beta'}) \in S_1$ with $\beta'<\beta$, i.e$.$ $\beta' \in \fp \cap \beta$, and $\pair{\alpha}{x} \in X_{\beta'}$. Then $\pair{\beta}{y} \in X_{\beta'}$ as well, so $(\pair{\beta}{y},\intv{\beta'}{\beta'}) \in S_1$. The reverse reasoning in which we start by considering $(\pair{\beta}{y},\intv{\beta'}{\beta'}) \in S_1$ also holds.

For proving that $s$ is a condition, we first observe that if $(\pair{\alpha}{x},\intv{\beta}{\gamma}) \in s$, then exactly one of the following holds: $\alpha \in \ep$, $\alpha \in \fd \setminus \fp$, and $\alpha \in \fp$. This is because it follows from the definitions that either $\alpha \in \ep$ or $\alpha \in \fp$, and we have that $\alpha \in \ep \cap \fp = \emptyset$ by \autoref{noninterference}. Moreover, exactly one of the following holds: $\beta \in \ep$ or $\beta \in \fp$, also by \autoref{noninterference}.


\autoref{cond-size}: This follows because the $S_k$'s are each defined from countably many parameters.

\autoref{cond-disjointness}: Fix $\pair{\alpha}{x}$, $\intv{\beta}{\gamma}
$, and $\intv{\beta'}{\gamma'}$. We consider the possible cases:

\begin{description}[style=unboxed,font=\normalfont,leftmargin=.2cm]

\item[$\alpha \in \ed$, $\beta,\beta' \in \ep$] If $i$ is large enough that $(\pair{\alpha}{x},\intv{\beta}{\gamma})$ and $(\pair{\alpha}{x},\intv{\beta'}{\gamma'})$ are in $s_i$, then this follows from the fact $s_i$ is a condition.


\item[$\alpha \in \ed$, $\beta \in \ep$, $\beta' \in \fp$] Then $\gamma' = \beta'$, so we can assume $\beta<\beta'$ and we want to show that $\gamma < \beta$. If this were not the case, then because of the fact that we have some $i<\omega$ with $(\pair{\alpha}{x},\intv{\beta}{\gamma}) \in s_i$,  and because $\pair{\alpha}{x}$ would witness $\beta' \in \fp$ in this case, it follows that $\beta \le \gamma$ would imply that $\beta' \notin \fp$, a contradiction.


\item[$\alpha \in \ed$, $\beta',\beta \in \fp$] Then $\gamma'=\beta'$ and $\gamma = \beta$, so there is nothing to deal with.

\item[$\alpha \in \fd \setminus \fp$] Then the clause follows from the fact that the intervals $\intv{\delta_n^\alpha}{\delta_{n+1}^\alpha-1}$ are disjoint.

\item [$\alpha \in \fp$, $\beta,\beta' \in \ep$] Then there is some $\tilde{x} \in \omega$ and $\tilde{\alpha} \in \ep$ such that $f_\alpha(x)=\pair{\tilde \alpha}{\tilde{x}}$. Then $(\pair{\alpha}{x},\intv{\beta}{\gamma}),(\pair{\alpha}{x},\intv{\beta'}{\gamma'}) \in S_2$, and so the fact that we have $(\pair{\tilde \alpha}{\tilde{x}},\intv{\beta}{\gamma}), (\pair{\tilde \alpha}{\tilde{x}},\intv{\beta'}{\gamma'}) \in s_i$ for large enough $i$ gives us the clause for this case.

\item[$\alpha \in \fp$, $\beta \in \ep$, $\beta' \in \fp$] We can assume $\beta<\beta'$ since for $\beta \ge \beta'$, $[\beta',\beta']$ is taken. It must be the case that $(\pair{\alpha}{x},\intv{\beta'}{\beta'}) \in S_3$ as witnessed by some $\pair{\tilde \alpha}{\tilde{x}}$ with $f_{\alpha}(x)=\pair{\tilde \alpha}{\tilde{x}}$ and moreover that $\pair{\tilde \alpha}{\tilde{x}} \in X_{\beta'}$. Then $\bit{\alpha}{x}{\beta}{\gamma} \in S_2$, so there is some $i<\omega$ such that $(\pair{\tilde \alpha}{\tilde{x}},\intv{\beta}{\gamma}) \in s_i$. Then if it were the case that $\beta' \le \gamma$, it would not be the case that $\pair{\tilde \alpha}{\tilde x} \in X_{\beta'}$.


\item[$\alpha \in \fp$, $\beta,\beta' \in \fp$] This is analogous to the case $\alpha \in \ed$, $\beta,\beta' \in \fp$ in the sense that we only need to consider whether or not $\beta = \beta'$.

\end{description}


\autoref{cond-oneomega1club}: If $\cf(\alpha) = \omega_1$ and $\alpha \in \dom(\bar s)$, then $\alpha \in \ed$. Then $\bar s$ inherits \autoref{cond-oneomega1club} from the fact that the $s_i$'s have \autoref{cond-oneomega1club}.

\autoref{cond-pointclosure}: If $\alpha \in \fd \setminus \fp$ this clause holds vacuously. For the other cases, suppose that $\{ (\pair{\alpha}{x} , \intv{\beta_n}{\gamma_n}) : i < \omega \} \subseteq s$ and $\beta^*=\sup_{n<\omega}\beta_n < \alpha$. We will subdivide the cases based on the minimal value $k$ for which infinitely many of these bits are taken from $S_k$.

\begin{description}[style=unboxed,font=\normalfont,leftmargin=.2cm]

\item[Inf$.$ many from $S_0$] Choose $i<\omega$ large enough so that $(\pair{\alpha}{x},\intv{\beta_n}{\gamma_n}) \in s_i$ for some $n$ in our infinite set. Then $(\pair{\alpha}{x},\intv{\beta_n}{\gamma_n}) \in s_i$ will hold for infinitely many $n$. Therefore we can see that either $\beta^* \in \ep$ or $\beta^* \in \fp$. If $\beta^* \in \ep$ then we are done. If $\beta^* \in \fp$, then this is witnessed by $\pair{\alpha}{x}$, so we have $(\pair{\alpha}{x},\intv{\beta^*}{\beta^*}) \in S_1 \subseteq s$.

\item[Inf$.$ many from $S_1$] Then we have infinitely many $(\pair{\alpha}{x},\intv{\beta_n}{\beta_n})$ where $\pair{\alpha}{x} \in X_{\beta_n}$. So for each such $n<\omega$ there is a sequence $\beta_n^j,\gamma_n^j$ converging to $\beta_n$ such that $(\pair{\alpha}{x},\intv{\beta_n^j}{\gamma_n^j}) \in S_0$, so it follows that $\beta^* \in \fp$ and that $\pair{\alpha}{x} \in X_{\beta^*}$. Therefore $(\pair{\alpha}{x},\intv{\beta^*}{\beta^*}) \in S_1 \subseteq s$.

\item[Inf$.$ many from $S_2$] Then let $\pair{\tilde{\alpha}}{\tilde{x}}$ be such that $f_\alpha(x) = \pair{\tilde{\alpha}}{\tilde{x}}$. Then we are saying that we have infinitely many $n<\omega$ such that $(\pair{\tilde{\alpha}}{\tilde{x}},\intv{\beta_n}{\gamma_n}) \in S_0$. Therefore $\beta^* \in \fp$, as witnessed by $\pair{\tilde{\alpha}}{\tilde{x}}$, and hence $(\pair{\alpha}{x},\intv{\beta^*}{\beta^*}) \in S_3 \subseteq s$.

\item[Inf$.$ many from $S_3$] Again choose $\pair{\tilde{\alpha}}{\tilde{x}}$ such that $f_\alpha(x) = \pair{\tilde{\alpha}}{\tilde{x}}$. We are saying that we have infinitely many $n<\omega$ such that $(\pair{\tilde{\alpha}}{\tilde{x}},\intv{\beta_n}{\beta_n}) \in S_3$ where $\pair{\tilde{\alpha}}{\tilde{x}}$ also witnesses that each $\beta_n$ is in $\fp$. This implies that $\pair{\tilde{\alpha}}{\tilde{x}}$ witnesses that $\beta^*$ is in $\fp$. Hence $(\pair{\alpha}{x},\intv{\beta^*}{\beta^*}) \in S_3 \subseteq s$.

\item[Inf$.$ many from $S_4$] We considered this case at the beginning of our discussion of \autoref{cond-pointclosure}.

\end{description}

%
%
%
%
%
%
%


%
%
%
%
%
%

\autoref{cond-domainclosure}: Suppose $\{(\pair{\alpha_i}{x_i},\intv{\beta_i}{\gamma_i}): i<\omega \} \subset s$ and $\alpha^* = \sup_{i<\omega}\alpha_i$. All possibilities are among the following:

\begin{description}[style=unboxed,font=\normalfont,leftmargin=.2cm]

\item[$\alpha^* \in \ed$] Then there is some $i<\omega$ such that $\alpha^* \in \dom s_i$, hence $\alpha^* \in \dom s$.

\item[$\alpha^* \in \fd \setminus \fp$] Then for all $n<\omega$ we have $(\pair{\alpha^*}{x},\intv{\delta_n^\alpha}{\delta_{n+1}^\alpha-1}) \in S_4 \subseteq s$ and so we have $\alpha^* \in \dom s$.

\item[$\alpha^* \in \fp$] Then there is some $\pair{\alpha}{y}$ witnessing $\alpha^* \in \fp$ and some $x$ such that $f_{\alpha^*}(x) = \pair{\alpha}{y}$. Choose some $(\pair{\alpha}{y},\intv{\beta}{\gamma}) \in S_0$ with $\gamma<\alpha^*$. Then $(\pair{\alpha^*}{x},\intv{\beta}{\gamma}) \in S_2 \subseteq s$.

\end{description}

\autoref{cond-unboundedness}: This is clear by inspection.


\autoref{cond-decisiveness}: If $\alpha \in \fd \setminus \fp$ then the clause holds vacuously.

For the other cases, suppose that $(\pair{\alpha}{x},\intv{\beta}{\gamma}) \in s$ and $\cf(\beta)=\omega$.

\begin{description}[style=unboxed,font=\normalfont,leftmargin=.2cm]

\item[$\alpha \in \ed,\beta \in \ep$] Then the clause holds because $(\pair{\alpha}{x},\intv{\beta}{\gamma}) \in s_i$ for large $i$ and $s_i$ is a condition.

\item[$\alpha \in \ed,\beta \in \fp$] Then we have $(\pair{\alpha}{x},\intv{\beta}{\gamma}) \in S_1$ where $\beta = \gamma$ and $\pair{\alpha}{x}$ witnesses $\beta \in \fp$. Then it follows that case (a) holds.

\item[$\alpha \in \fp,\beta \in \ep$]

In this case we have $(\pair{\alpha}{x},\intv{\beta}{\gamma}) \in S_2$, meaning that $f_\alpha(x)=\pair{\tilde{\alpha}}{\tilde{x}}$ for some $\pair{\tilde{\alpha}}{\tilde{x}} \in X_\alpha$, and that $(\pair{\tilde{\alpha}}{\tilde{x}},\intv{\beta}{\gamma}) \in S_0$. Hence if $i<\omega$ is large enough that $(\pair{\tilde{\alpha}}{\tilde{x}},\intv{\beta}{\gamma}) \in s_i$, then the clause is witnessed by the fact that $s_i$ is a condition and $\Coh(s_i,\tilde{\alpha},\tilde{x},\alpha,x)$ holds.

\item[$\alpha \in \fp,\beta \in \fp$] In this case we have $(\pair{\alpha}{x},\intv{\beta}{\gamma}) \in S_2$, meaning that $\beta = \gamma$, $f_\alpha(x) = \pair{\tilde{\alpha}}{\tilde{x}}$, and $\pair{\tilde{\alpha}}{\tilde{x}}$ also witnesses $\beta \in \fp$. Hence there are $\beta_i,\gamma_i$ converging to $\beta$ such that $(\pair{\tilde{\alpha}}{\tilde{x}},\intv{\beta_i}{\gamma_i}) \in  S_0$ for $i<\omega$. Therefore $(\pair{\alpha}{x},\intv{\beta_i}{\gamma_i}) \in  S_2$ for $i<\omega$. This implies that case (a) holds.

\end{description}

\autoref{cond-coherence}: Referring to the statement in \autoref{newforcing}, fix $\pair{\alpha}{x}$ and $\beta$ such that the hypothesis holds for $(\pair{\alpha}{x},\intv{\beta}{\gamma}) \in s$ for some $\gamma$.

\begin{description}[style=unboxed,font=\normalfont,leftmargin=.2cm]


\item[$\alpha \in \fd \setminus \fp$] Then coherence holds vacuously because there are no limit points to consider.

\item[$\alpha \in \ed, \beta \in \ep$] Since $\beta \in \ep$, there is some $i<\omega$ such that $s_i \Vdash ``\beta \in \lim \dot{C}_\alpha^x$''. Therefore there is some $y$ witnessing that $\Coh(s_i,\alpha,x,\beta,y)$ holds. This also holds for $i' \ge i$ by the definition of $\le_\S$. This means that if $(\pair{\alpha}{x},\intv{\beta'}{\gamma'}) \in s$ for $\gamma'<\beta$, then there are two possibilities. The first is that $\beta' \in \ep$, and hence $(\pair{\alpha}{x},\intv{\beta'}{\gamma'}) \in s_j$ for some $j \ge i$, in which $\Coh(s_j,\alpha,x,\beta,y)$ implies that $(\pair{\beta}{y},\intv{\beta'}{\gamma'}) \in s_j \subseteq s$. The second possibility is that $\beta' \in \fp$. Then it must be the case that $(\pair{\alpha}{x},\intv{\beta'}{\gamma'}) \in S_1$, meaning that $\beta'=\gamma'$ and that $\pair{\alpha}{x} \in X_{\beta'}$. Then this plus the fact that $\Coh(s_i,\alpha,x,\beta,y)$ holds for large $i$ implies that $\pair{\beta}{y} \in X_{\beta'}$. Therefore $(\pair{\beta}{y},\intv{\beta'}{\beta'}) \in S_1$.

Now suppose $(\pair{\beta}{y},\intv{\beta'}{\gamma'}) \in s$ and consider the same two possibilities. If $\beta' \in \ep$, then if $i$ is large enough that $s_i \Vdash ``\beta \in \lim \dot{C}_\alpha^x$'' and $(\pair{\beta}{y},\intv{\beta'}{\gamma'}) \in s_i$, then $\Coh(s_i,\alpha,x,\beta,y)$ implies $(\pair{\alpha}{x},\intv{\beta'}{\gamma'}) \in s_i$. If $\beta' \in \fp$ then $(\pair{\beta}{y},\intv{\beta'}{\gamma'}) \in S_1$, implying $\beta'=\gamma'$ and $\pair{\beta}{y} \in X_{\beta'}$, so $\pair{\alpha}{x} \in X_{\beta'}$, hence $(\pair{\alpha}{x},\intv{\beta'}{\gamma'}) \in S_1$.


\item[$\alpha \in \ed$, $\beta \in \fp$] Then $\pair{\alpha}{x}$ witnesses $\beta \in \fp$, so let $y$ be such that $f_\beta(y)=\pair{\alpha}{x}$. Then we can argue that $\Coh(s,\alpha,x,\beta,y)$ holds. If $(\pair{\alpha}{x},\intv{\beta'}{\gamma'}) \in s$ and $\beta' \in \ep$, then $(\pair{\alpha}{x},\intv{\beta'}{\gamma'}) \in S_0$, and hence $(\pair{\beta}{y},\intv{\beta'}{\gamma'}) \in S_2$. If $(\pair{\alpha}{x},\intv{\beta'}{\gamma'}) \in s$ and $\beta' \in \fp$, then $\gamma'=\beta'$ and $\pair{\alpha}{x} \in X_{\beta'}$, so $(\pair{\beta}{y},\intv{\beta'}{\gamma'}) \in S_3$. The other inclusion follows the same reasoning.

\item[$\alpha \in \fp$, $\beta \in \ep$] Then there is some $\pair{\tilde \alpha}{\tilde x}$ witnessing $\alpha \in \fp$ with $f_\alpha(x)=\pair{\tilde \alpha}{\tilde x}$. All bits $(\pair{\alpha}{x},\intv{\beta}{\gamma}) \in s$ such that $\beta \in \ep$ are in $S_2$ and witnessed by $(\pair{\tilde \alpha}{\tilde x},\intv{\beta}{\gamma}) \in S_0$. Therefore there is some $i<\omega$ with $s_i \Vdash ``\beta \in \lim \dot{C}^{\tilde x}_{\tilde \alpha}$''. Let $y$ be such that $\Coh(s_i,\tilde{\alpha},\tilde{x},\beta,y)$ holds. We can then argue that $\Coh(s,\alpha,x,\beta,y)$ holds. Suppose that $\beta',\gamma'<\beta$ and $(\pair{\alpha}{x},\intv{\beta'}{\gamma'}) \in \tilde S$ with $\beta' \in \ep$. Then $(\pair{\tilde \alpha}{\tilde x},\intv{\beta'}{\gamma'}) \in S_0$, so $(\pair{\beta}{y},\intv{\beta'}{\gamma'}) \in S_0$. If instead $\beta' \in \fp$, then $(\pair{\alpha}{x},\intv{\beta'}{\beta'}) \in S_3$ and $\pair{\tilde \alpha}{\tilde x} \in X_{\beta'}$, which implies $\pair{\beta}{y} \in X_{\beta'}$, so $(\pair{\beta}{y},\intv{\beta'}{\beta'}) \in S_3$. Again, the reverse reasoning applies if we start with the premise that $(\pair{\beta}{y},\intv{\beta'}{\gamma'}) \in s$.

\item[$\alpha \in \fp$, $\beta \in \fp$] Then $(\pair{\alpha}{x},\intv{\beta}{\beta}) \in S_3$, so there is some $\pair{\tilde \alpha}{\tilde x} \in X_\alpha \cap X_\beta$ such that $f_\alpha(x)=\pair{\tilde \alpha}{\tilde x}$. Let $y$ be such that $f_\beta(y)=\pair{\tilde \alpha}{\tilde x}$. We argue that $\Coh(s,\alpha,x,\beta,y)$ holds. Suppose $\bit{\alpha}{x}{\beta'}{\gamma'} \in s$. If $\beta' \in \ep$, then $(\pair{\alpha}{x},\intv{\beta'}{\gamma'}) \in S_2$, so $(\pair{\tilde \alpha}{\tilde x},\intv{\beta'}{\gamma'}) \in S_0$, so $(\pair{\beta}{y},\intv{\beta}{\gamma'}) \in S_2$. If $\beta' \in \fp$, then $\beta'=\gamma'$ and $(\pair{\alpha}{x},\intv{\beta'}{\beta'}) \in S_3$, implying $\pair{\tilde \alpha}{\tilde x} \in X_{\beta'} \cap X_\alpha$, i.e$.$ $\pair{\tilde \alpha}{\tilde x} \in X_{\beta'} \cap X_\beta \cap X_\alpha$, implying $(\pair{\beta}{y},\intv{\beta'}{\beta'}) \in S_3$ also. Once more, we can reverse the reasoning.


\end{description}

This finishes the proof of \autoref{closure}.\end{proof}

The next task is to prove that $\S$ preserves $\aleph_2$.


\begin{lemma}\label{squaremaster} If $s \in \S$ and $M \prec K$ is a basic model with $s,\S \in M$, then $s$ is strongly $(M,\S)$-generic.\end{lemma}

Therefore \autoref{basicprop} implies:

\begin{corollary}\label{aleph2-pres} Under $\CH$, $\S$ preserves $\aleph_2$.\end{corollary}

\begin{proof}[Proof of \autoref{squaremaster}] Let $\delta = M \cap \aleph_2$. Suppose $s' \le s$ and  let $s^*$ consist of all $(\pair{\alpha}{x},\intv{\beta}{\gamma}) \in s'$ such that $\alpha,\beta,\gamma < \delta$. Then $s^* \in M$ by the countable closure of $M$. We can see that $s^* \in \S$ since it is an initial segment of $s'$, the supremum of whose domain is a limit ordinal. Let $s^{**} \le s^*$ where $s^{**} \in M$. We will construct $s''$ such that $s'' \le s',s^{**}$.

For each $\alpha \in \dom(s') \setminus \delta$, $x \in \omega$, let $X_{\pair{\alpha}{x}}$ be the set of $\pair{\beta}{y}$ with $\beta<\delta$ such that $s' \Vdash ``\beta \in \lim \dot{C}_\alpha^x \cap \delta$'', and such that $\Coh(s',\alpha,x,\beta,y)$ holds. Let $S_0$ be the set
\begin{multline*}
\{(\pair{\alpha}{x},\intv{\beta'}{\gamma'}): \alpha \in \dom(s')\setminus \delta, x \in  \omega, \exists \pair{\beta}{y} \in X_{\pair{\alpha}{x}},\\ (\pair{\beta}{y},\intv{\beta'}{\gamma'}) \in s^{**}\}.
\end{multline*}
We then argue that $s'' := s^{**} \cup s' \cup S_0$ is a condition by verifying the more substantial clauses from \autoref{newforcing}.



\autoref{cond-disjointness}: Fix $(\pair{\alpha}{x},\intv{\beta}{\gamma})$ and $(\pair{\alpha}{x},\intv{\beta'}{\gamma'})$ for which we will prove the clause. The clause follows from the fact that $s^{**}$ and $s'$ are conditions in $\S$ if $\alpha < \delta$ or if both bits are in $s'$. Suppose that $(\pair{\alpha}{x},\intv{\beta}{\gamma}) \in s'$ and $(\pair{\alpha}{x},\intv{\beta'}{\gamma'}) \in S_0 \setminus s'$ as witnessed by $\pair{\bar \beta}{y} \in X_{\pair{\alpha}{x}}$. Since $\gamma'<\bar \beta$, we can assume without loss of generality that $\beta < \bar \beta$, from which it follows that $\gamma<\bar \beta$ since $s' \Vdash ``\bar \beta \in \lim \dot{C}^x_\alpha$''. Then $(\pair{\bar \beta}{y},\intv{\beta}{\gamma}) \in s^*$ and $(\pair{\bar \beta}{y},\intv{\beta'}{\gamma'}) \in s^{**}$, meaning that both bits are in $s^{**}$. Hence it must be that either $\beta = \beta'$ or $\intv{\beta}{\gamma} \cap \intv{\beta'}{\gamma'} = \emptyset$. The remaining case, in which both $(\pair{\alpha}{x},\intv{\beta}{\gamma})$ and $(\pair{\alpha}{x},\intv{\beta'}{\gamma'})$ are both in $S_0 \setminus s'$, is similar.



\autoref{cond-coherence}: Fix $\pair{\alpha}{x}$ and $\beta$ as in the statement of \autoref{cond-coherence} in \autoref{newforcing}. If $\alpha<\delta$, then if $\bit{\alpha}{x}{\beta'}{\gamma'} \in s^{**} \cup s' \cup S_0$ for some $x,\beta,\gamma$, then $\bit{\alpha}{x}{\beta'}{\gamma'} \in s^{**}$, so we have what we need because $s^{**} \in \S$ and satisfies \autoref{cond-coherence}. If $\alpha \in \aleph_2 \setminus \delta$ and $\beta \ge \delta$, then we have what we need because $s' \in \S$. Observe that if $\beta<\delta$, $\alpha \in \aleph_2 \setminus \delta$, and $s' \Vdash ``\beta \in \lim \dot{C}_\alpha$'', then we have $\beta \in \dom(s^{**})$, so the clause holds for the remaining case due to bits in $S_0$.\end{proof}

\begin{lemma}\label{addingsquare} $\S$ adds a $\square(\aleph_2,\aleph_0)$-sequence.\end{lemma}

\begin{proof} Suppose $G$ is $\S$-generic over $V$. We know from \autoref{closure} and \autoref{aleph2-pres} that $\S$ preserves cardinals and cofinalities up to and including $\aleph_2$. In $V[G]$ we let $\vec{\mathcal{C}} = \seq{\mathcal{C}_\alpha}{\alpha \in \Lim(\aleph_2)}$ be defined so that $\mathcal{C}_\alpha = \{C_\alpha^x:x \in \omega\}$ and $C_\alpha^x = \dot{C}_\alpha^x[G]$. That the $C_\alpha^x$'s are clubs was handled by \autoref{objects-are-clubs}.


First, we show that $\vec{\mathcal{C}}$ is coherent. Suppose that $s \in \S$ forces that $\beta$ is a limit point of $\dot{C}_\alpha^x$, and suppose for contradiction that there is no $y \in \omega$ and no $s' \le s$ forcing that $``\dot{C}_\alpha^x \cap \beta = \dot{C}_\beta^y"$. Then we find an $\le_\S$-decreasing sequence $\seq{s_i}{i<\omega}$ of complete conditions such that $s_i \Vdash ``\dot{C}_\alpha^x \cap \beta \ne \dot{C}_\beta^i"$, and we find a lower bound $\bar s$ of this sequence using the countable closure of $\S$. Then $\bar s$ violates \autoref{cond-coherence}.

Next, we show that $\vec{\mathcal{C}}$ does not have a thread in $V[\S]$, using the usual genericity argument. Suppose $s$ forces that $\dot D$ is a closed unbounded subset of $\aleph_2$. Let $\bar \alpha$ and $s' \le s$ be such that $s'$ forces $\bar \alpha$ to be the $\omega\th$ point of $\dot D$.  Build an $\S$-decreasing sequence of complete conditions $\seq{s_i}{i<\omega}$ below $s'$ and an increasing sequence of ordinals $\seq{\alpha_i}{i<\omega}$ above $\bar \alpha$ as follows: Let $\alpha_0 = \bar \alpha$ and $s_0 = s'$. Given $s_i$ and $\alpha_i$, $s_{i+1} \le s_i$ and $\alpha_{i+1}$ will be chosen such that $s_{i+1} \Vdash ``\alpha_{i+1} \in \lim \dot{D}"$ and such that $\alpha_{i+1}> \max s_i$ and $\max s_{i+1}>\alpha_i$. Then let $\alpha^*$ be the supremum of the $\alpha_i$'s and let $s^*$ be a lower bound of the $s_i$'s forcing that $\dot{\mathcal{C}}_{\alpha^*}$ consists only of $\omega$-sequences---this is specifically possible from the argument in \autoref{closure}, where $\alpha^*$ would be in $\fd \setminus (\ed \cup \fp)$ in that argument. Therefore $s^* \Vdash ``\dot D \cap \alpha^* \notin \dot{\mathcal{C}}_{\alpha^*}$''.\end{proof}


\section{Threads}\label{sec-newthreading}

Next we introduce the threading forcing $\T$. Of course, a $\square(\aleph_2,\aleph_0)$-sequence is defined as such because it has no thread, but our ability to force a thread allows us to make use of large cardinals from the ground model. This is also how the threading forcing for Jensen's method works. The goal of this section is to show that we can use our own $\T$ to lift weakly compact embeddings of the form $j:\mathcal{M} \to \mathcal{N}$.


\begin{definition}\label{newthreading} Let $G_\S$ be $\S$-generic over $V$ and work in $V[G_\S]$. For all $\alpha \in \Lim(\aleph_2)$ and $x \in \omega$, let $C_\alpha^x = \dot{C}_\alpha^x[G_\S]$.

Define $\mathfrak{T}$ to be the set of closed intervals $\intv{\beta}{\gamma}$ such that $-1 \le \beta \le \gamma < \aleph_2$. Then we define a poset of $\T$ of conditions $t \subset \mathfrak{T}$ such that the following holds:

\begin{enumerate}

\item\label{tcond-ctble} The set $t$ is countable.

\item\label{tcond-disjointness} If $\intv{\beta}{\gamma}$, $\intv{\beta'}{\gamma'} \in t$, then either $\beta = \beta'$ or $\intv{\beta}{\gamma} \cap \intv{\beta'}{\gamma'} = \emptyset$.

\item\label{tcond-pointclosure} If $\{  \intv{\beta_i}{\gamma_i} : i < \omega \} \subseteq t$ and $\beta^*=\sup_{i<\omega}\beta_i$, then
it follows that $ \intv{\beta^*}{\gamma^*} \in t$ for some $\gamma^*$.


\item\label{tcond-decisiveness} If $\intv{\beta}{\gamma} \in t$, then one of the following holds:

\begin{enumerate}
\item $\cf(\beta) = \omega_1$.
\item $\beta = \beta'+1$ and there is some $\bar \beta \le \beta'$ such that $\intv{\bar \beta}{\beta'} \in t$.
\item $\cf(\beta)=\omega$ and there exists a sequence of ordinals $\seq{\beta_i}{i<\omega}$ converging to $\beta$ such that $\forall i < \omega, \exists \gamma_i,\intv{\beta_i}{\gamma_i} \in t$.
\item $\cf(\beta)=\omega$ and there is some $\bar \beta < \beta$ and some sequence of ordinals $\seq{\gamma_i}{i<\omega}$ converging to $\beta$ such that $\forall i < \omega, \intv{\bar \beta}{\gamma_i} \in t$.
\end{enumerate}


\item\label{tcond-top} There is some $\alpha<\aleph_2$ such that $\intv{\alpha}{\alpha} \in t$ and such that for all $\beta > \alpha$, there is no $\gamma$ such that $\intv{\beta}{\gamma} \in t$. Either $\cf(\alpha) = \omega_1$ or there is some $\seq{\beta_n}{n<\omega}$ cofinal in $\alpha$ such that for all $n<\omega$, $\intv{\beta_n}{\gamma} \in t$ for some $\gamma$.

\item\label{tcond-isathread} Suppose that $\intv{\beta}{\gamma} \in t$ and either $\cf(\beta) = \omega_1$ or else $\{\beta'<\beta: \exists \gamma',\intv{\beta'}{\gamma'} \in t\}$ is unbounded in $\beta$. Then there some $x \in \omega$ such that for all $\beta' \le \gamma' < \beta$, $\intv{\beta'}{\gamma'} \in t$ implies both $\beta' \in C_\beta^x$ and $C_\beta^x \cap \interval[open left]{\beta'}{\gamma'} = \emptyset$.


\end{enumerate}


We have $t' \le_\T t$ if and only if:

\begin{enumerate}
\item $t' \supseteq t$;
\item If $\pair{\alpha}{x}$ is such that for $\beta,\gamma<\alpha$, $\intv{\beta}{\gamma} \in t$ implies $\beta \in C_\alpha^x$ and $C_\alpha^x \cap \interval[open left]{\beta}{\gamma} = \emptyset$, then for $\beta,\gamma<\alpha$, $\intv{\beta}{\gamma} \in t'$ implies $\beta \in C_\alpha^x$ and $C_\alpha^x \cap \interval[open left]{\beta}{\gamma} = \emptyset$.

\end{enumerate}
\end{definition}

We use the following conventions for $t \in \T$:

\begin{itemize}

\item $\pos(t) = \{\beta \ge 0:\exists \gamma, \intv{\beta}{\gamma} \in t\}$.

\item $\max(t)$ is the largest element of $\pos(t)$.


\end{itemize}

\begin{definition} If $s \in \S$, let
\[
\thr(s) = \{( \{\intv{\alpha}{\alpha}\} \cup \{ \intv{\beta}{\gamma}:(\pair{\alpha}{x},\intv{\beta}{\gamma}) \in s \}): \alpha \in \dom(s), x \in \omega\}.
\]
We say that $\pair{\alpha}{x}$ \emph{witnesses} that $t \in \thr(s)$ if $\intv{\alpha}{\alpha} \in t$ and for $\beta,\gamma < \alpha$, $\intv{\beta}{\gamma} \in t$ if and only if $\bit{\alpha}{x}{\beta}{\gamma} \in s$.\end{definition}




Now we can connect the notion of $\thr(s)$ for $s \in \S$ with $\T$ if we are working in $V$.

\begin{proposition}\label{whythethreads} The following are true given $s \in \S$:

\begin{enumerate}
\item If $t \in \thr(s)$ is witnessed by $\pair{\alpha}{x}$ then $s \Vdash ``t \in \dot{\T}"$. 
\item If $s \Vdash ``\dot t \in \dot{\T}"$, then there is some $s' \le s$ and some $t' \in \thr(s')$ such that $s' \Vdash ``t' \le \dot{t}"$.
\end{enumerate}\end{proposition}

\begin{proof} \emph{(1)} is clear, with the appropriate $\alpha \in \dom s$ and $x \in \omega$ witnessing \autoref{tcond-isathread} from \autoref{newthreading}. \emph{(2)} works as follows: Let $s^* \le s$, $\alpha$, and $x \in \omega$ be such that $s^* \Vdash ``\alpha = \max(\dot t)$ and $\pair{\alpha}{x}$ witnesses \autoref{tcond-isathread} of \autoref{newthreading}''. Then use the countable closure of $\S$ to find $s' \le s^*$ such that $s' \Vdash ``\dot t = t^+ \in V"$ and such that for all $\intv{\beta}{\gamma} \in t^+$, either $\bit{\alpha}{x}{\beta}{\gamma} \in s'$ or else there is some $\bit{\alpha}{x}{\beta'}{\gamma'} \in s'$ such that $\beta \ne \beta'$ and $\intv{\beta}{\gamma} \cap \intv{\beta'}{\gamma'} \ne \emptyset$. Then it must be the case that if $\intv{\beta}{\gamma} \in t^+$ then $\bit{\alpha}{x}{\beta}{\gamma} \in s'$ because otherwise the definition of $\le_{\dot{\T}}$ would be violated. Then let $t'$ be such that $\intv{\alpha}{\alpha} \in t'$ and such that for all $\beta,\gamma<\alpha$, $\intv{\beta}{\gamma} \in t'$ if and only if $\bit{\alpha}{x}{\beta}{\gamma} \in s'$. Then $s'$ and $t'$ witness \emph{(2)}.\end{proof}



The next step is to introduce our version of the two-step iterations that appear in Jensen's method of forcing squares.


\begin{definition}\label{twostepdef} We let $D(\S \ast \T)$ refer to the set of pairs $(s,\dot t) \in \S \ast \T$ such that:

\begin{enumerate}
\item\label{stcond-threads} $s \Vdash ``\dot t = \check t"$ for some $t \in \thr(s)$.
\item\label{stcond-maxima} $s \Vdash ``\max(\dot t)=\max(s) \ge \sup(\out(s))"$.
\item $s$ is complete.

\end{enumerate}\end{definition}

If $(s,\check t) \in D(\S \ast \T)$, we will most often write $(s, \check t)$ as $(s,t)$.

\begin{lemma}\label{twostepclosure} The set $D(\S \ast \T)$ is dense in $\S \ast \T$ and countably closed. Moreover, if a sequence $\seq{(s_i,t_i)}{i<\omega}$ is $\le_{\S \ast \T}$-decreasing, then it has a lower bound $(\bar s,\bar t)$ such that $\bar s$ is a parsimonious lower bound of $\seq{s_i}{i<\omega}$.\end{lemma}

\begin{proof} We will prove each claim separately. The ``moreover'' part of the statement will be clear from the proof.

\begin{claim}\label{twostepclosure-density} $D(\S \ast \T)$ is dense in $\S \ast \T$.\end{claim}

Suppose $(s,\dot t) \in \S \ast \T $. Let $s'$ and $t'$ witness \autoref{whythethreads} with respect to $(s,\dot t)$. Then choose $s'' \le s'$, $t'' \in \thr(s'')$, and a large enough $\alpha^*$ such that $(s'',\check{t}'') \le (s',\check{t}')$ and such that $\alpha^* = \max(t'' ) = \max(s'') \ge \sup(\out(s''))$ as follows: Let $\beta^* = \sup\{\gamma+1:\exists \beta, \intv{\beta}{\gamma} \in t'\}$ and choose $\alpha^*$ such that $\max(s') \cup \sup(\out(s')) < \alpha^* $. Let $\seq{\alpha_n}{n<\omega}$ be a sequence with supremum $\alpha^*$ such that $\alpha_0 = \beta^*$ and $\alpha_n$ is a successor ordinal for $n>0$. Let $s''$ be
\[
s' \cup \{\bit{\alpha^*}{x}{\alpha_n}{\alpha_{n+1}-1}:x,n \in \omega\} \cup \{\bit{\alpha^*}{x}{\beta}{\gamma}:\intv{\beta}{\gamma} \in t', x \in \omega\}.
\]
Let $t''$ be such that $\intv{\alpha^*}{\alpha^*} \in t''$ and such that for all $\beta,\gamma<\alpha^*$, $\intv{\beta}{\gamma} \in t''$ if and only if $\bit{\alpha^*}{x}{\beta}{\gamma} \in s''$ for all $x \in \omega$.


\begin{claim}\label{twostepclosure-closure} $D(\S \ast \T)$ is countably closed.\end{claim}

Let $\seq{(s_i,t_i)}{i<\omega}$ be $\le_{\S \ast \T }$-decreasing in $D(\S \ast \T)$. Let $s^*$ be a parsimonious lower bound of $\vec s = \seq{s_i}{i<\omega}$. We use some definitions with the goal of eventually constructing $\bar t$:

\begin{itemize}

\item $T_0 = \bigcup_{i<\omega}t_i $,
\item $\ep(t) = \{\beta : \exists \gamma, \intv{\beta}{\gamma} \in T_0\}$,
\item $\fp(t) = \limp(\ep(t)) \setminus \ep(t)$,
\item $t^* = T_0 \cup \{ \intv{\beta}{\beta} : \beta \in \fp(t)\}$.

\end{itemize}

Now let $\bar \alpha$ be the supremum of $\ep(t) \cup \fp(t)$. If there is some $i$ such that $\bar \alpha \in \dom(s_i)$, in other words $\bar \alpha \in \ep(t)$, then it will be the case that $\pos(t^*) = \pos(\tilde t)$ for some $\tilde t \in \thr(\bar s)$.  Then we can let $\bar s =s^*$ and let $\bar t$ be $\tilde t$. Let us therefore assume for the rest of the proof that for all $i<\omega$, $\bar \alpha \notin \dom s_i$. Then we will define a lower bound $(\bar s,\bar t)$ of $\seq{(s_i,t_i)}{i<\omega}$ by modifying $s^*$ (we will \emph{not} necessarily have that $\bar{s}$ and $s^*$ are comparable) and we will argue that $(\bar{s},\bar{t}) \in D(\S \ast \T )$.

We will define the modification $\bar s$ as follows:

\begin{itemize}

\item If $\alpha \ne \bar \alpha$, then $(\pair{\alpha}{x},\intv{\beta}{\gamma}) \in \bar{s}$ if and only if $(\pair{\alpha}{x},\intv{\beta}{\gamma}) \in s^*$.

\item $(\pair{\bar \alpha}{x+1},\intv{\beta}{\gamma}) \in \bar{s}$ if and only if $(\pair{\bar \alpha}{x},\intv{\beta}{\gamma}) \in s^*$.

\item $(\pair{\bar \alpha}{0},\intv{\beta}{\gamma}) \in \bar{s}$ if $\intv{\beta}{\gamma} \in t^*$ and $\gamma < \bar \alpha$.

\end{itemize}

We need to argue that $\bar s \in \S$. It suffices to show that \autoref{cond-coherence}, i.e$.$ coherence, from \autoref{newforcing} holds for bits of the form $\bit{\bar\alpha}{0}{\beta'}{\gamma'}$ i.e$.$ for coherence of $\dot{C}^0_{\bar{\alpha}}$. Let $\beta$ be such that either $\cf(\beta) = \omega_1$ or else $\cf(\beta)=\omega$ and $\{\beta'<\beta: \exists \gamma',(\pair{\alpha}{x},\intv{\beta'}{\gamma'}) \in s\}$ is unbounded in $\beta$. Observe that there is some $\beta^* \in (\beta,\alpha)$ such that $\beta^* \in \ep(t)$ and either Case (a) or Case (c) of \autoref{tcond-decisiveness} of \autoref{newthreading} holds for $\beta^*$. Let $x$ be such that $\pair{\beta^*}{x}$ witnesses \autoref{tcond-isathread} of \autoref{newthreading}. Then $\beta^* \in \dom s^*$ and $s^* \Vdash ``\beta \in \lim \dot{C}_{\beta^*}^x$'', so let $y$ witness coherence of $\dot{C}_{\beta^*}^x$ at $\beta$. Then $y$ witnesses coherence of $\dot{C}^0_{\bar \alpha}$ at $\beta$. Coherence for $\dot{C}^0_x$ with $x>0$ is given by the ``shift'' in the definition of $\bar{s}$.




We let $\bar t$ be defined so that $\intv{\bar \alpha}{\bar \alpha} \in \bar t$ and so that for all $\beta,\gamma < \bar \alpha$, $\intv{\beta}{\gamma} \in \bar t$ if and only if $\bit{\bar \alpha}{0}{\beta}{\gamma} \in \bar s$. We claim that $(\bar{s},\bar{t}) \in D(\S \ast \T)$. It can be checked that $\bar{s}$ is a lower bound of $\seq{s_i}{i<\omega}$: Observe that $\bar \alpha = \max(\bar s)= \max (s^*)$ because $s^*$ was chosen to be parsimonious and because \autoref{stcond-maxima} and \autoref{stcond-threads} from \autoref{twostepdef} hold for $(s_i,t_i)$ for all $i<\omega$. In particular, the fact that $\sup(\out(s_i)) \le \max(s_i)$ for all $i<\omega$ implies that $\fp(\vec s)\subseteq \max(s^*)$. We only needed to guarantee that we would have $\bar t \in \thr(\bar s)$.\end{proof}

We therefore see that $\T$ preserves $\aleph_1$ over $V[\S]$ since it is a factor of a countably distributive iteration. Moreover, we can expand the proof from \autoref{squaremaster} to see:

\begin{proposition}\label{threadcardpres} $\S \ast \T$ preserves $\aleph_2$ over $V$.\end{proposition}

\begin{proof} Specifically, suppose that $s \in \S$, $t \in \T$, and $M \prec K$ is a basic model with $s,\S \ast \T \in M$, and $\delta = M \cap \aleph_2$. Then it can be argued that $(s,t \cup \intv{\delta}{\delta})$ is strongly $(M,\S \ast \T)$-generic. Then apply \autoref{basicprop}.
\end{proof}

Although $\T$ preserves cardinals, it does not preserve the canonical $\square(\aleph_2,\aleph_0)$-sequence added by $\S$. This is the primary function of $\T$ in some sense, despite its usefulness for lifting embeddings (which will be established shortly).

\begin{proposition} If $G_\S \ast G_\T$ is $\S \ast \T$-generic over $V$, and $\mathcal{D} = \{\beta < \aleph_2:\exists \gamma, \intv{\beta}{\gamma} \in G_\T\}$, then $\mathcal{D}$ is a thread of the $\square(\aleph_2,\aleph_0)$-sequence $\mathcal{C}$ derived from $G_\S$.\end{proposition}



The following lemma shows how we will use a weakly compact cardinal together with $\S$ by lifting embeddings, using $\T$ in a crucial way. In the context of Jensen's method for forcing squares, we would use a generic condition argument for $j(\S)$. However, when we have lifted $j:\mathcal{M}[\Col(\aleph_1,<\!\kappa)] \to \mathcal{N}[\Col(\aleph_1,<\!j(\kappa))]$ and are working in $V[\Col(\aleph_1,<\!j(\kappa)$], we have $|\S|=\aleph_1$, yet the conditions in both $\S$ and $j(\S)$ are countable. Therefore we use a factorization argument instead.

\begin{lemma}\label{basiclifting} Suppose the following:
\begin{itemize}
\item $\mathfrak{C}=\Col(\aleph_1,<\kappa)$ is the L{\'e}vy collapse and $G_\mathfrak{C}$ is $\mathfrak{C}$-generic over $\bar V$.
\item $j:\mathcal{M} \to \mathcal{N}$ is a weakly compact embedding with critical point $\kappa$.
\item $j:\mathcal{M}[G_\mathfrak{C}] \to \mathcal{N}[j(G_\mathfrak{C})]$ is the usual lifted embedding given by $\dot x_{G_\mathfrak{C}} \mapsto j(\dot x)_{j(G_\mathfrak{C})}$.
\item $V[j(G_\mathfrak{C})]$ contains generics $G_\S$ and $G_\T$ for $\S$ and $\T$ respectively.
\end{itemize}
Then in $V[j(G_\mathfrak{C})]$, $j(\S)$ is forcing equivalent to $\S \ast \T \ast \S'$ where $\S'$ is a countably closed forcing over $V[j(G_\mathfrak{C})]$.\end{lemma}

Observe that the fourth premise can be fulfilled using the Absorption Lemma.



\begin{proof} First we describe $\S'$ and some relevant dense subsets of $\S$ and $j(\S)$. Working in $V[G_{j(\mathfrak{C})}]$, $j(\mathfrak{S})$ (where $\mathfrak{S}$ is from \autoref{newforcing}) is of course equal to the set of bits $(\pair{\alpha}{x},\intv{\beta}{\gamma})$ such that $\alpha<j(\kappa)$, $x \in \omega$, $\beta \le \gamma < j(\kappa)$. Let $C_\alpha^x = \dot{C}_\alpha^x[G_\S]$ for all $\pair{\alpha}{x} \in \Lim(\aleph_2) \times \omega$ and let $\mathcal{D} = \{\beta < \aleph_2:\exists \gamma, \intv{\beta}{\gamma} \in G_\T\}$.

Now let $\S'$ consist of $s \subseteq j(\mathfrak{S}) \setminus \mathfrak{S}$ such that \autoref{cond-size} through \autoref{cond-decisiveness} from \autoref{newforcing} hold, such that its ordering is defined like $\le_\S$, and moreover two additional clauses hold, the first of which is a modification of \autoref{cond-coherence} from \autoref{newforcing}:

\begin{enumerate}
\setcounter{enumi}{9}

\item\label{xcond-newcoherence} Suppose that $(\pair{\alpha}{x},\intv{\beta}{\gamma}) \in s$ and either $\cf(\beta) = \omega_1$ or else $\{\beta'<\beta: \exists \gamma',(\pair{\alpha}{x},\intv{\beta'}{\gamma'}) \in s\}$ is unbounded in $\beta$. Then:

\begin{enumerate}[(a)]

\item If $\beta > \kappa$, then there is some $y \in \omega$ such that for all $\beta' \le \gamma' < \beta$, $(\pair{\alpha}{x},\intv{\beta'}{\gamma'}) \in s$ if and only if $(\pair{\beta}{y},\intv{\beta'}{\gamma'}) \in s$.

\item If $\beta < \kappa$, then there is some $y \in \omega$ such that for all $\beta' \le \gamma' < \beta$, if $(\pair{\alpha}{x},\intv{\beta'}{\gamma'}) \in s$ then $\beta' \in C^y_\beta$ and $C^y_\beta \cap \interval[open left]{\beta'}{\gamma'} = \emptyset$.


\item If $\beta = \kappa$, then for all $\beta',\gamma' < \beta$, if $(\pair{\alpha}{x},\intv{\beta'}{\gamma'}) \in s$, then $\beta' \in \mathcal{D}$ and $\mathcal{D} \cap \interval[open left]{\beta'}{\gamma'} = \emptyset$.


\end{enumerate}

\item\label{xcond-cutoff} For all limit ordinals $\alpha$ and all $x \in \omega$ such that $\exists \tilde \beta, \tilde \gamma, (\pair{\alpha}{x},\intv{\tilde \beta}{\tilde \gamma}) \in s$, one of the following holds:

\begin{enumerate}[(a)]

\item $(\pair{\alpha}{x},\intv{\kappa}{\gamma}) \in s$ for some $\gamma$.

\item There exists a non-empty finite sequence $\beta_0 < \beta_1 < \ldots < \beta_k$ such that the following hold:

\begin{enumerate}[(i)]

\item For all $\ell \in \{0,1,\ldots,k\}$, $(\pair{\alpha}{x},\intv{\beta_\ell}{\gamma}) \in s$ for some $\gamma$;

\item Either $\cf(\beta_0) = \omega_1$ or there are sequences $\seq{\beta_n}{n<\omega}$, $\seq{\gamma_n}{n<\omega}$ with $\sup_{n<\omega}\beta_n = \beta$ and $(\pair{\alpha}{x},\intv{\beta_n}{\gamma_n})\in s$ for all $n<\omega$;

\item For all $\beta_\ell$ with $\ell>0$, either $\beta_{\ell}=\beta_{\ell-1}+1$, or $\beta_\ell = \bar \beta + 1$ and $\bit{\alpha}{x}{\beta_{\ell-1}}{\bar \beta} \in s$, or there is a sequence $\seq{\gamma_n}{n<\omega}$ with $\sup_{n<\omega}\gamma_n = \beta_\ell$ and $(\pair{\alpha}{x},\intv{\beta_{\ell-1}}{\gamma_n}) \in s$ for all $n<\omega$.

\item $(\pair{\alpha}{x},\intv{\beta_\ell}{\gamma}) \in s$ for some $\gamma \ge \kappa$.

\end{enumerate} 
\end{enumerate} 

%
%
%
%
\end{enumerate}

Let $D(j(\S))$ be the dense subset of $s \in j(\S)$ such that \autoref{xcond-cutoff} holds for $s$. Now we give the isomorphism $\Psi: \{(s,t,s'): s \in \S,(s,t) \in D(\S \ast \T), s' \in \S'\} \to D(j(\S))$. Specifically,
\[
\Psi(s,t,s') = s \cup s' \cup \{(\pair{\kappa}{x},\intv{\beta}{\gamma}): x \in \omega, \intv{\beta}{\gamma} \in t \}.
\]
Observe that the definition of $\Psi$ is where we use \autoref{cond-oneomega1club} from \autoref{newforcing}. We are also using the fact that $V[G_{j(\mathfrak{C})}] \models ``\kappa \in j(\kappa) \cap \cof(\omega_1)$''. It is immediate that $\Psi$ is order-preserving. Moreover, $\Psi$ is a bijection because it has a natural inverse.

%


Now we will show that $\S'$ is countably closed in $V[G_{j(\mathfrak{C})}]$. Suppose $\vec s = \seq{s_i}{i<\omega}$ is a descending sequence in $\S'$. Define $\ed$, $\fd$, $\ep$, $\fp$ as in \autoref{kindsofpoints}, and let $\bar s$ be a parsimonious lower bound of $\vec s$ defined exactly as in \autoref{closure}. So $\bar s \in j(\S)$. We know that $\bar s$ satisfies \autoref{cond-size} through \autoref{cond-decisiveness} from \autoref{newforcing}. We will prove \autoref{xcond-newcoherence} after \autoref{xcond-cutoff}, since the former depends on the proof of the latter.

\autoref{xcond-cutoff}: The case where $\alpha \in \fd \setminus (\ep \cup \fp)$ is trivial. Given $\alpha \in \ed$, $x \in \omega$, there is some $i$ such that $\alpha \in \dom(s_i)$. Then the fact that $s_i$ satisfies \autoref{xcond-cutoff} is enough to imply that $\bar s$ satisfies \autoref{xcond-cutoff}. If $\alpha \in \fp$, then there is some $\pair{\alpha'}{x'}$ witnessing $\alpha \in \fp$, and we get \autoref{xcond-cutoff} from the $s_i$ for $i$ large enough that $\alpha' \in \dom(s_i)$.

\autoref{xcond-newcoherence}: For the purpose of verifying this clause, when we say that $y$ witnesses coherence of $\dot{C}^x_\alpha$ at $\beta$ for $\beta \le \kappa$, we are referring to the $y$ mentioned in Case (b) or Case (c) of the clause. 

Now fix $\alpha \in \Lim(j(\kappa) \setminus (\kappa+1))$ and $x \in \omega$ as in the statement of \autoref{xcond-newcoherence}. We have already handled the case where $\beta > \kappa$ in \autoref{closure}. Therefore we can assume that $\beta<\kappa$ for the rest of the proof. We break the proof into cases.

\begin{description}[style=unboxed,leftmargin=.2cm]

\item[$\alpha \in \ep, \beta \in \ep$, for large $i$ Case (a) of \autoref{xcond-cutoff} holds for $s_i$ w.r.t$.$ $\pair{\alpha}{x}$] We have that $\beta \in \lim \mathcal{D}$, so there is some $y \in \omega$ such that $\mathcal{D} \cap \beta =  C_\beta^y$. Then $y$ witnesses coherence of $\dot{C}_x^\alpha$ at $\beta$.

\item[$\alpha \in \ep, \beta \in \fp$, for large $i$ Case (a) of \autoref{xcond-cutoff} holds for $s_i$ w.r.t$.$ $\pair{\alpha}{x}$] Observe that since $\S \ast \T$ has a countably closed dense subset, it follows that $\T$ is countably distributive, so $G_\T$ is countably closed, hence $\beta \in \lim \mathcal{D}$. This then follows the reasoning of the previous case.


\item[$\alpha \in \ep, \beta \in \ep$, for large $i$ Case (b) of \autoref{xcond-cutoff} holds for $s_i$ w.r.t$.$ $\pair{\alpha}{x}$] This means that there is some $\beta^+$ such that for large $i$, $s_i \Vdash ``\beta^+ = \max(\lim \dot{C}_\alpha^x \cap \kappa)"$, and that $\beta \le \beta^+$. If $\beta = \beta^+$ and $y$ witnesses coherence of $\dot{C}_\alpha^x$ at $\beta^+$, then we are done with this case. If $\beta<\beta^+$, $y$ witnesses coherence of $\dot{C}_\alpha^x$ at $\beta^+$, and $y'$ is such that $C^y_{\beta^+} \cap \beta = C_\beta^{y'}$ (since $\beta \in \lim C_{\beta^+}^y$), then $y'$ witnesses coherence of $\dot{C}_\alpha^x$ at $\beta$.

\item[$\alpha \in \ep, \beta \in \fp$, for large $i$ Case (b) of \autoref{xcond-cutoff} holds for $s_i$ w.r.t$.$ $\pair{\alpha}{x}$] This is like the previous case, noting that if $\beta<\beta^+$ as written there, then we would still have $\beta \in \lim C_{\beta^+}^y$.

\item[$\pair{\tilde \alpha}{\tilde x}$ witnesses $\alpha \in \fp, \beta \in \ep$, for lg$.$ $i$ Case (a) of \autoref{xcond-cutoff} holds for $s_i$, $\pair{\tilde \alpha}{\tilde x}$] We have $\beta \in \lim \mathcal{D}$, so there is some $y \in \omega$ such that $\mathcal{D} \cap \beta = C_\beta^y$. So $y$ witnesses coherence of $\dot{C}^{\tilde x}_{\tilde \alpha}$ at $\beta$, hence it witnesses coherence of $C^x_\alpha$ at $\beta$.

\end{description}

The remaining cases all consider $\alpha \in \fp$ and take $\pair{\tilde \alpha}{\tilde x}$ witnessing this. The reasoning depends on whether Case (a) or Case (b) holds for $s_i$ with respect to $\pair{\tilde \alpha}{\tilde x}$ for large $i$. The arguments are analogous to the cases already discussed.\end{proof}
 
%

Now we can use $\T$ to show that stronger squares do not hold. Since the bulk of the work was done with \autoref{basiclifting}, and the remaining lemmas are variations of standard arguments, we will handle the rest with a light amount of detail.

\begin{lemma}\label{threadpres1} If $W$ is a ground model and $\mathcal{C}$ is a $\square(\aleph_2,<\aleph_0)$-sequence in $W[\S]$, then $\T$ does not thread $\mathcal{C}$.\end{lemma}

\begin{proof} This lemma has analogs for Jensen-style posets (see \cite{AST-Magidor}, Lemma 4.5). Let $\dot D$ be a name for a thread for $\mathcal{C}$ added over $W[\S]$ by $\T$ and work in $W$. Build a descending sequence $\seq{s_i}{i<\omega}$ of conditions in $\S$, sequences $\seq{t_i^j}{j < \omega}$ in $\T$ for $i<\omega$, and an increasing sequence of ordinals $\seq{\gamma_i}{i<\omega}$ such that the following hold:

\begin{itemize}
\item $(s_i,t_i^j) \in D(\S \ast \T)$ for all $i,j<\omega$.

\item $t_i^j = t_0^j$ for all $j<i$.

\item For all $i,j<\omega$, $\max t_0^j = \max t_i^j$.

\item $(s_{i+1},t_0^{i+1})$ and $(s_{i+1},t_{i+1}^{i+1})$ decide ``$\gamma_i \in \dot D$'' differently. 
\end{itemize}

Then let $\bar s$ be a lower bound of $\seq{s_i}{i<\omega}$ and let $\bar{t}_i$ be a lower bound of $\seq{t_i^j}{j < \omega}$ for all $i<\omega$, and also let $\bar \gamma$ be a lower bound of $\seq{\gamma_i}{i<\omega}$. Then each $(\bar s,\bar{t}_i)$ gives a different possibility for $C \in \mathcal{C}_{\bar \gamma}$, contradicting the fact that $\mathcal{C}$ is a $\square(\aleph_2,<\aleph_0)$-sequence.
\end{proof}

\begin{lemma}\label{threadpres2} If $\mathcal{C}$ is a $\square(\lambda,\aleph_0)$-sequence and $\P$ is countably closed, then $\P$ does not thread $\mathcal{C}$.\end{lemma}

\begin{proof} Stronger versions of this lemma appear elsewhere (\cite{Hayut-LambieHanson2017}), but we sketch an argument for completeness.


Let $\dot D$ be a name for a thread supposedly added by $\P$. Build a tree $\seq{p_x}{x \in 2^{<\omega}}$ and a collection of ordinals $\seq{\gamma_x}{x \in 2^{<\omega}}$ such that:

\begin{itemize}
\item $x \sqsubseteq y$ implies $p_y \le p_x$,
\item $p_{x {}^\frown 0}$ and $p_{x {}^\frown 1}$ decide ``$\gamma_x \in \dot D$'' differently for some $\gamma_x$ above $\sup_{|y|<x}\gamma_y$.
\end{itemize}

Then for all $X \in 2^\omega$, let $p_X$ be the lower bound of $\seq{p_{X \rest n}}{n<\omega}$. Let $\gamma = \sup_{x \in 2^{\omega}}$. Then the $p_X$'s give $2^\omega$-many possibilities for $C \in \mathcal{C}_\gamma$, which is a contradiction of the fact that $\mathcal{C}$ is a $\square(\lambda,\aleph_0)$-sequence.\end{proof}

\begin{proposition}\label{threadpres3} If $\mathcal{C}$ is a $\square_{\mu,\kappa}$-sequence and $\P$ preserves $\mu^+$, then $\P$ does not thread $\mathcal{C}$.\end{proposition}

\begin{proof} This is basically a restatement of \autoref{whynothread}.\end{proof}

\begin{corollary}\label{failure1} If $\kappa$ is weakly compact, then $V[\Col(\aleph_1,<\kappa)][\S] \models \neg \square(\aleph_2,<\aleph_0)$.\end{corollary}

\begin{proof} Let $M$ be a $\kappa$-sized transitive model containing $\S$ and a supposed $\Col(\aleph_1,<\kappa) \ast \S$-name $\dot{\mathcal{C}}$ for a $\square(\aleph_2,<\aleph_0)$-sequence. Let $j:M \to N$ be a weakly compact embedding with critical point $\kappa$, and let $G_\mathfrak{C} \ast G_\S \ast G_\T$ be $\Col(\aleph_1,<\kappa) \ast \S \ast \T$-generic over $V$. Since $\S \ast \T$ is countably closed, the quotient $\mathfrak{R}:=\Col(\aleph_1,<j(\kappa))/G_\mathfrak{C} \ast G_\S \ast G_\T)$, is forcing-equivalent to a countably closed forcing (see the section on absorption in \cite{Handbook-Cummings}). Let $G_\mathfrak{R}$ be generic for $\mathfrak{R}$ and work in $V[G_\mathfrak{C} \ast G_\S \ast G_\T \ast G_\mathfrak{R}]$ to use \autoref{basiclifting} to lift $j$ to $j:M[G_\mathfrak{C}][G_\S] \to N[j(G_\mathfrak{C})][j(G_\S)]$ by forcing with a $\S'$-generic $G_{\S'}$. Use the lift to define a thread $D$ for $\mathcal{C}:=\dot{\mathcal{C}}[G_\mathfrak{C} \ast G_\S]$ by taking an element from the $\kappa\th$ level of $j(\mathcal{C})$. \autoref{threadpres1} shows that $\T$ could not have added $D$, and \autoref{threadpres2} show that neither $\mathfrak{R}$ nor $\S'$ could have added $D$. Therefore the thread $D$ already exists in $V[G_\mathfrak{C} \ast G_\S]$, and so $\mathcal{C}$ is not actually a $\square(\aleph_2,<\aleph_0)$-sequence.\end{proof}

\begin{corollary}\label{failure2} If $\kappa$ is weakly compact then $V[\Col(\aleph_1,<\kappa)][\S] \models \neg \square_{\aleph_1,\aleph_0}$.\end{corollary}

\begin{proof} A Mahlo cardinal would be sufficient (see \cite{AST-Magidor}, Theorem 4.4). But assume we have a weakly compact cardinal and define a thread $D$ as in \autoref{failure1}. Then \autoref{threadpres3} and \autoref{threadcardpres} together imply that $\T$ could not have added $D$. Then \autoref{threadpres2} shows that the rest of the extension could not have added $D$.\end{proof}

Hence we have finished proving \autoref{maintheorem}.

\subsection*{Acknowledgement}


I would also like to thank the anonymous referees for carefully reading the manuscript and dealing with very rough initial versions. Thank you to Heike Mildenberger for helping with the revision.

\end{document}